\documentclass{amsart}

\usepackage{amsmath}
\usepackage{amsxtra}
\usepackage{amscd}
\usepackage{amsthm}
\usepackage{amsfonts}
\usepackage{amssymb}
\usepackage{mathtools}
\usepackage{eucal}
\usepackage{graphicx}
\usepackage{mathrsfs}
\usepackage{color}
\usepackage{caption}
\usepackage{enumerate}
\usepackage{faktor}
\usepackage{bbm}
\usepackage{hyperref}
\usepackage[T1]{fontenc} 
\usepackage[utf8]{inputenc} 
\usepackage{tikz}
\usepackage{tikz-cd}
\usepackage[margin=1in]{geometry}
\usepackage{xcolor}
\usepackage{enumitem}

\hypersetup{
  colorlinks = true,
  urlcolor = cyan,
}

\theoremstyle{plain}
\newtheorem{theorem}{Theorem}[section]
\newtheorem{lemma}[theorem]{Lemma}
\newtheorem{proposition}[theorem]{Proposition}

\newtheorem{claim}[theorem]{Claim}

\theoremstyle{definition}
\newtheorem{definition}[theorem]{Definition}
\newtheorem{conjecture}[theorem]{Conjecture}

\theoremstyle{remark}
\newtheorem{remark}[theorem]{Remark}

\numberwithin{equation}{section}

\newcommand{\E}{\mathbb{E}}
\newcommand{\F}{\mathbb{F}}
\newcommand{\N}{\mathbb{N}}
\newcommand{\Q}{\mathbb{Q}}

\newcommand{\Z}{\mathbb{Z}}
\renewcommand{\P}{\mathbb{P}}
\newcommand{\ol}{\overline}

\newcommand{\Aut}{{\mathrm{Aut}}}
\newcommand{\Hom}{{\mathrm{Hom}}}
\DeclareMathOperator{\GL}{GL}
\DeclareMathOperator{\cok}{cok}
\DeclareMathOperator{\Sur}{Sur}
\DeclareMathOperator{\Cl}{Cl}
\DeclareMathOperator{\M}{M}

\newcommand{\verts}[1]{\mathopen{}\left\lvert#1\right\rvert\mathclose{}}
\newcommand{\floor}[1]{\mathopen{}\left\lfloor#1\right\rfloor\mathclose{}}
\newcommand{\ceil}[1]{\mathopen{}\left\lceil#1\right\rceil\mathclose{}}

\makeatletter
\AtBeginDocument
 {
    \def\@thm#1#2#3{%
      \ifhmode
        \unskip\unskip\par
      \fi
      \normalfont
      \trivlist
      \let\thmheadnl\relax
      \let\thm@swap\@gobble
      \let\thm@indent\indent 
      \thm@headfont{\scshape}
      \thm@notefont{\fontseries\mddefault\upshape}%
      \thm@headpunct{.}
      \thm@headsep 5\p@ plus\p@ minus\p@\relax
      \thm@space@setup
      #1
      \@topsep \thm@preskip               
      \@topsepadd \thm@postskip           
      \def\dth@counter{#2}%
      \ifx\@empty\dth@counter
        \def\@tempa{%
          \@oparg{\@begintheorem{#3}{}}[]%
        }%
      \else
        \H@refstepcounter{#2}%
        \hyper@makecurrent{#2}%
        \let\Hy@dth@currentHref\@currentHref
        \AddToHookNext{para/begin}{\MakeLinkTarget*{\Hy@dth@currentHref}}%
        \def\@tempa{%
          \@oparg{\@begintheorem{#3}{\csname the#2\endcsname}}[]%
        }%
      \fi
      \@tempa
    }%
  \dth@everypar={%
    \@minipagefalse \global\@newlistfalse
    \@noparitemfalse
    \if@inlabel
      \global\@inlabelfalse
      \begingroup \setbox\z@\lastbox
       \ifvoid\z@ \kern-\itemindent \fi
      \endgroup
      \unhbox\@labels
    \fi
    \if@nobreak \@nobreakfalse \clubpenalty\@M
    \else \clubpenalty\@clubpenalty \everypar{}%
    \fi
  }%
}

\newcommand{\starpar}[1]{%
  \par\vspace{\abovedisplayskip}%
  \noindent
  \makebox[1cm][l]{$(\star)$}%
  \parbox[t]{\dimexpr\linewidth-2cm}{#1}%
  \par\vspace{\belowdisplayskip}%
}

\title{Universality for cokernels of partially random integral matrices}

\author{Isaac Rajagopal}
\address{Massachusetts Institute of Technology}
\email{isaacraj@mit.edu}

\date{\today}


\begin{document}

\begin{abstract}
    Given any $\varepsilon > 0$, let $M(n)$ be a random $n \times (n+u)$ matrix over $\mathbb{Z}_p$, with all entries independent and $\varepsilon$-balanced (lying in each residue class mod $p$ with probability at most $1-\varepsilon$). Wood \cite{Wood15} proved that as $n \to \infty$ the distribution of $\mathrm{cok}(M(n))$ approaches Cohen and Lenstra's conjectured distribution of class groups. Given $\alpha,\beta >0$ such that $\alpha + \beta <1$, we prove that the distribution of $\mathrm{cok}(M(n))$ still approaches the Cohen--Lenstra distribution even if we weaken the hypothesis by allowing up to $\alpha n$ entries per column and up to $\beta n$ entries per row of $M(n)$ to not be $\varepsilon$-balanced. We also weaken the independence condition by allowing certain types of dependence between the entries of each column. In addition, we prove that, for any $\delta > 0$, the cokernels of random band matrices of width $\log(n)^{1+\delta}$ with $\varepsilon$-balanced entries in the band and arbitrary entries outside of it will also approach the Cohen--Lenstra distribution, which answers a question of Kang--Lee--Yu.
\end{abstract}

\maketitle

\section{Introduction}
The Cohen--Lenstra heuristics \cite{CL84} are conjectures which describe the distribution of Sylow $p$-subgroups of class groups of quadratic number fields. For a finite abelian group $G$, let $G_p$ denote its Sylow $p$-subgroup. 
\begin{conjecture}[Cohen--Lenstra \cite{CL84}]\label{conjCL}
    Let $S_X^+$ (resp. $S_X^-$) denote the set of positive (resp. negative) fundamental discriminants $D$ with $|D| < X$. Let $p$ be an odd prime and $B$ a finite abelian $p$-group. Then \[\lim_{X \to \infty} \frac{\#\{D \in S_X^{\pm}: \Cl(\Q(\sqrt{D}))_p \simeq B\}}{|S_X^\pm|} = \frac{\prod_{k=1}^\infty (1-p^{-k-u})}{|B|^{u}|\Aut(B)|},\] where $u = 0$ if we consider $S_X^-$ and $u =1$ if we consider $S_X^+$.
\end{conjecture}

This question naturally leads to the study of cokernels of random matrices over $\Z_p$, because $\Cl(\Q(\sqrt{D}))_p$ can be written as a cokernel $\cok(R_D) \coloneq \Z_p^n/R_D\Z_p^{n+u}$ for some matrix $R_D \in \M_{n \times (n+u)}(\Z_p)$ (cf. \cite{Wood15}). Friedman and Washington \cite{FW89} proved that, if $H(n)$ is a random matrix valued in $\M_{n \times n}(\Z_p)$ with independent random entries distributed according to the Haar measure on $\Z_p$, then as $n \to \infty$ the distribution of $\cok(H(n))$ converges to the distribution in Conjecture~\ref{conjCL} with $u = 0$. Wood \cite{Wood15} showed that a much weaker condition on the distribution of each entry, called being $\varepsilon$-balanced, was sufficient.

\begin{definition}
    Let $p$ be a prime, $\varepsilon >0$ a real number, and $X$ a random variable valued in $\Z_p$ or a finite quotient of $\Z_p$. Then $X$ is \emph{$\varepsilon$-balanced} if $\P(X \equiv t \pmod{p}) \leq 1- \varepsilon$ for all $t \in \Z/p\Z$. Otherwise, $X$ is \emph{$\varepsilon$-degenerate}.
\end{definition}

\begin{theorem}{\cite{Wood15}}\label{thmWood}
    Let $p$ be a prime, $u \geq 0$ an integer, $\varepsilon > 0$ a real number, and $B$ a finite abelian $p$-group. For each positive integer $n$, let $M(n)$ be a random matrix valued in $\M_{n \times (n+u)}(\Z_p)$ with independent entries $M(n)_{ij}$ which are all $\varepsilon$-balanced. Then \[\lim_{n \to \infty} \P(\cok(M(n)) \simeq B) = \frac{\prod_{k=1}^\infty (1-p^{-k-u})}{|B|^{u}|\Aut(B)|}.\]
\end{theorem}

Unfortunately, Theorem \ref{thmWood} does not resolve Conjecture \ref{conjCL} because the matrix $R_D\in \M_{n \times (n+u)}(\Z_p)$ for a random fundamental discriminant $D$ does not necessarily have $\varepsilon$-balanced or independent entries. In this paper, we will weaken the assumptions in Theorem \ref{thmWood}, by allowing many entries $M(n)_{ij}$ to be $\varepsilon$-degenerate (See Section \ref{sectindependent}) and additionally allowing dependence between the entries in each column (See Section~\ref{sectdependent}). In Section \ref{sectindependent}, we have two main results: Theorem \ref{mainthm} describes cokernels of matrices where $51\%$ of the entries in every row and every column are $\varepsilon$-balanced and Theorem \ref{theorembandmatrix} describes cokernels of band matrices of width $\log(n)^{1+\delta}$ with $\varepsilon$-balanced entries in the band and arbitrary entries outside it, and is a universal version of \cite{M24}. Section~\ref{sectdependent} contains Theorem \ref{thmdependent}, which generalizes Theorem \ref{mainthm} by allowing dependence between the entries of each column of $M(n)$ such that the entries in each column become independent after a basis change. These theorems show that many new random matrix ensembles belong to the universality class of matrices whose cokernels approach the Cohen--Lenstra distribution.

\subsection{Matrices with some \texorpdfstring{$\varepsilon$-}{}degenerate entries and independent entries.}\label{sectindependent} We state our general theorem about matrices with independent entries where many entries are $\varepsilon$-degenerate.

\begin{figure}
    \centering
\begin{minipage}{0.3\textwidth}
    \centering
    \begin{tikzpicture}[scale=0.8]
    \useasboundingbox (-0.8, -1.2) rectangle (3.3, 3.8);
  \draw[thick]
    (-0.1, 3)
    .. controls (-0.2, 2.6) .. (-0.2, 2.3)
    -- (-0.2, 0.7)
    .. controls (-0.2, 0.4) .. (-0.1, 0.0);

  \draw[thick]
    (3.1, 3)
    .. controls (3.2, 2.6) .. (3.2, 2.3)
    -- (3.2, 0.7)
    .. controls (3.2, 0.4) .. (3.1, 0.0);

  \fill[blue!25]
    (0.0, 0) -- (0.0,2.01) -- (1.98, 2.01) -- (1.98,3) --(3,3)--(3,0) -- cycle;

  \draw[<->, >=Stealth, thin]
    (0, 3.2) -- (1.98, 3.2)
    node[midway, above] {$0.66n$};

 \draw[<->, >=Stealth, thin]
    (-0.4,2.01) -- (-0.4, 3)
    node[midway, left] {$0.33n$};

\node at (1.5, -1) {$(a)$};

\end{tikzpicture}
    \end{minipage}
\hfill
\begin{minipage}{0.3\textwidth}
    \centering
    \begin{tikzpicture}[scale=0.8]
    \useasboundingbox (-0.9, -1.2) rectangle (4.2, 3.8);
  \draw[thick]
    (-0.1, 3)
    .. controls (-0.2, 2.6) .. (-0.2, 2.3)
    -- (-0.2, 0.7)
    .. controls (-0.2, 0.4) .. (-0.1, 0.0);

  \draw[thick]
    (3.1, 3)
    .. controls (3.2, 2.6) .. (3.2, 2.3)
    -- (3.2, 0.7)
    .. controls (3.2, 0.4) .. (3.1, 0.0);

\fill[blue!25]
    (0,3) -- (0.0,1.47) -- (1.47, 1.47) -- (1.47,0) -- (3,0) -- (3,1.53) -- (1.53,1.53) -- (1.53,3) -- cycle;

  \draw[<->, >=Stealth, thin]
    (1.47, -0.2) -- (3, -0.2)
    node[midway, below] {$0.51n$};

  \draw[<->, >=Stealth, thin]
    (0, 3.2) -- (1.53, 3.2)
    node[midway, above] {$0.51n$};

 \draw[<->, >=Stealth, thin]
    (3.4, 0) -- (3.4, 1.53)
    node[midway, right] {$0.51n$};

 \draw[<->, >=Stealth, thin]
    (-0.4,1.47) -- (-0.4, 3)
    node[midway, left] {$0.51n$};

\node at (1.5, -1) {$(b)$};

\end{tikzpicture}
    \end{minipage}
\hfill
\begin{minipage}{0.3\textwidth}
    \centering
    \begin{tikzpicture}[scale=0.8]
    \useasboundingbox (-0.5, -1.2) rectangle (4.5, 3.8);
  \draw[thick]
    (-0.1, 3)
    .. controls (-0.2, 2.6) .. (-0.2, 2.3)
    -- (-0.2, 0.7)
    .. controls (-0.2, 0.4) .. (-0.1, 0.0);

  \draw[thick]
    (3.1, 3)
    .. controls (3.2, 2.6) .. (3.2, 2.3)
    -- (3.2, 0.7)
    .. controls (3.2, 0.4) .. (3.1, 0.0);

  \fill[blue!25]
    (0.0, 2.22) -- (2.22, 0.0) -- (3, 0.0) -- (3, 0.78) --
    (0.78, 3) -- (0.0, 3) -- cycle;

 \fill[blue!25]
    (0.0, 0.78) -- (0.0,0.0)--(0.78,0) -- cycle;

 \fill[blue!25]
    (3, 2.22) -- (3,3)--(2.22,3) -- cycle;

  \draw[<->, >=Stealth, thin]
    (0.0, -0.2) -- (0.78, -0.2)
    node[midway, below] {$0.26n$};

 \draw[<->, >=Stealth, thin]
    (2.22, -0.2) -- (3, -0.2)
    node[midway, below] {$0.26n$};

 \draw[<->, >=Stealth, thin]
    (3.4, 0) -- (3.4, 0.78)
    node[midway, right] {$0.26n$};

  \draw[<->, >=Stealth, thin]
    (3.4, 2.22) -- (3.4, 3)
    node[midway, right] {$0.26n$};

\node at (1.5, -1) {$(c)$};

\end{tikzpicture}
    \end{minipage}
    \caption{These are three examples of random matrices $M(n)$ valued in $\M_{n\times n}(\Z_p)$ whose cokernels approach the Cohen--Lenstra distribution by Theorem \ref{mainthm}. Blue regions represent $\varepsilon$-balanced entries and white regions represent entries with unrestricted distributions. For example, the entries in the white regions may all be fixed to zero.}\label{fig1}
\end{figure}
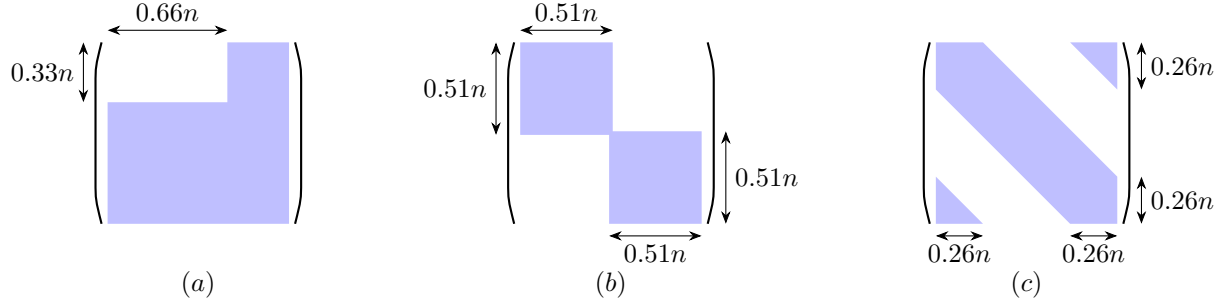

\begin{theorem}\label{mainthm}
    Let $p$ be a prime, $u \geq 0$ an integer, $\varepsilon > 0$ a real number, and $B$ a finite abelian $p$-group. Let $\alpha, \beta >0$ be real numbers satisfying $\alpha + \beta < 1$. For each positive integer $n$, let $M(n)$ be a random matrix valued in $\M_{n \times (n+u)}(\Z_p)$ with independent entries $M(n)_{ij}$, with at most $\alpha n$ $\varepsilon$-degenerate entries in each column and at most $\beta n$ $\varepsilon$-degenerate entries in each row. Then 
    \[\lim_{n \to \infty} \P(\cok(M(n)) \simeq B) = \frac{\prod_{k=1}^\infty (1-p^{-k-u})}{|B|^u\verts{\Aut(B)}}.\]
\end{theorem}

\begin{figure}
    \centering
    \centering
\begin{minipage}{0.4\textwidth}
    \centering
    \begin{tikzpicture}[scale=0.8]
    \useasboundingbox (-1.5, -1.2) rectangle (4.5, 3.8);
  \draw[thick]
    (-0.1, 3)
    .. controls (-0.2, 2.6) .. (-0.2, 2.3)
    -- (-0.2, 0.7)
    .. controls (-0.2, 0.4) .. (-0.1, 0.0);

  \draw[thick]
    (3.1, 3)
    .. controls (3.2, 2.6) .. (3.2, 2.3)
    -- (3.2, 0.7)
    .. controls (3.2, 0.4) .. (3.1, 0.0);

  \fill[blue!25]
    (0.0, 0) -- (0.0,1.8) -- (2.1, 1.8) -- (2.1,3) --(3,3)--(3,0) -- cycle;

  \draw[<->, >=Stealth, thin]
    (0, 3.2) -- (2.1, 3.2)
    node[midway, above] {$0.7n$};

 \draw[<->, >=Stealth, thin]
    (-0.4,1.8) -- (-0.4, 3)
    node[midway, left] {$0.4n$};

\node at (1.5, -1) {$(d)$};
\node at (1.05, 2.4) {$0$};
\end{tikzpicture}
    \end{minipage}
\begin{minipage}{0.4\textwidth}
    \centering
    \begin{tikzpicture}[scale=0.8]
    \useasboundingbox (-1.5, -1.2) rectangle (4.5, 3.8);
  \draw[thick]
    (-0.1, 3)
    .. controls (-0.2, 2.6) .. (-0.2, 2.3)
    -- (-0.2, 0.7)
    .. controls (-0.2, 0.4) .. (-0.1, 0.0);

  \draw[thick]
    (3.1, 3)
    .. controls (3.2, 2.6) .. (3.2, 2.3)
    -- (3.2, 0.7)
    .. controls (3.2, 0.4) .. (3.1, 0.0);

\fill[blue!25]
    (0,3) -- (0.0,1.5) -- (1.5, 1.5) -- (1.5,0) -- (3,0) -- (3,1.5) -- (1.5,1.5) -- (1.5,3) -- cycle;

  \draw[<->, >=Stealth, thin]
    (1.5, -0.2) -- (3, -0.2)
    node[midway, below] {$0.5n$};

  \draw[<->, >=Stealth, thin]
    (0, 3.2) -- (1.5, 3.2)
    node[midway, above] {$0.5n$};

 \draw[<->, >=Stealth, thin]
    (3.4, 0) -- (3.4, 1.5)
    node[midway, right] {$0.5n$};

 \draw[<->, >=Stealth, thin]
    (-0.4,1.5) -- (-0.4, 3)
    node[midway, left] {$0.5n$};

\node at (1.5, -1) {$(e)$};

\node at (0.75, 0.75) {$0$};
\node at (2.25, 2.25) {$0$};

\end{tikzpicture}
    \end{minipage}
    \caption{These are two examples of random matrices $M(n)$ valued in $\M_{n\times n}(\Z_p)$ whose cokernels do \emph{not} approach the Cohen--Lenstra distribution. Blue regions represent $\varepsilon$-balanced entries and white regions represent entries which are all fixed to zero.}\label{fig2}
\end{figure}
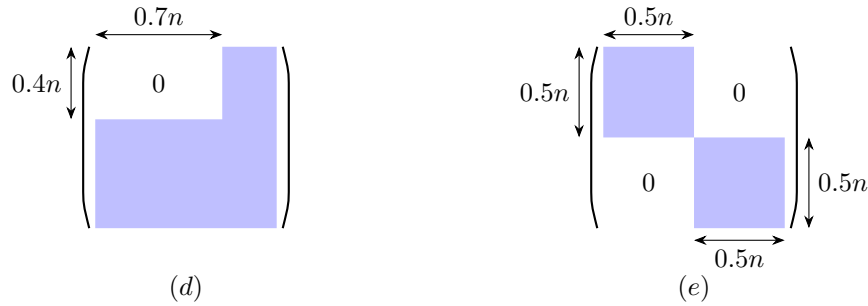

Figure \ref{fig1} shows three examples of random matrices $M(n)$ which satisfy the assumptions of Theorem \ref{mainthm} with various values of $\alpha$ and $\beta$: in (a) we can take $\alpha = 0.33$ and $\beta = 0.66$ and in (b) and (c), we can take $\alpha = \beta =0.49$. So, the cokernels of all of these matrices approach the Cohen--Lenstra distribution.

Figure \ref{fig2} shows two examples of random matrices $M(n)$ whose cokernels do \emph{not} approach the Cohen--Lenstra distribution which illustrate that the requirement that $\alpha + \beta <1$ in Theorem \ref{mainthm} is tight. The matrix in (d) can never have full rank so its cokernel will not approach the Cohen--Lenstra distribution. The cokernel of the matrix in (e) is given by the direct sum of the cokernels of two $(n/2) \times (n/2)$ matrices with all entries $\varepsilon$-balanced, so it will have a different distribution than the Cohen--Lenstra distribution. These matrices satisfy the assumptions of Theorem~\ref{mainthm} except that in (d) we have $\alpha = 0.4$ and $\beta = 0.7$ and in (e) we have $\alpha = \beta = 0.5$, so $\alpha + \beta  \geq 1$ in both examples.

If we prescribe specific locations of the $\varepsilon$-balanced entries, then we can prove Cohen--Lenstra universality for random matrix ensembles with many fewer $\varepsilon$-balanced entries than in Theorem \ref{mainthm}. M\'esz\'aros' \cite{M24} studied band matrices, which have a particularly convenient set of distinguished entries. Let $B(n)$ be a random matrix in $\M_{n\times n}(\Z_p)$ with $B(n)_{ij}$ distributed according to the Haar measure on $\Z_p$ if $|i-j| \leq w_n$ and $B(n)_{ij} = 0$ otherwise. Then, M\'esz\'aros \cite{M24} found that $\cok (B(n))$ approaches the Cohen--Lenstra distribution if and only if $\lim_{n \to \infty} w_n - \log_p (n) = \infty$. In Theorem \ref{theorembandmatrix}, we generalize this by allowing the entries in the band to be arbitrary $\varepsilon$-balanced variables, and allowing the entries outside the band to be fully arbitrary, partially answering \cite[Problem 22]{M24}. To do this, we have to require $w_n \geq \log(n)^{1+\delta}$ for some $\delta > 0$. Figure \ref{fig3} illustrates that Theorem~\ref{theorembandmatrix} allows the entries in the band to be arbitrary $\varepsilon$-balanced variables rather than Haar-uniform and allows the entries outside the band to be arbitrary rather than zero, while M\'esz\'aros' \cite{M24} work applies to matrices with slightly narrower bands. 

\begin{figure}
    \centering
    \begin{minipage}{0.45\textwidth}
    \centering
    \begin{tikzpicture}[scale=1.2]
    \useasboundingbox (-0.5, -1.2) rectangle (4.5, 3.8);
  \draw[thick]
    (-0.1, 3)
    .. controls (-0.2, 2.6) .. (-0.2, 2.3)
    -- (-0.2, 0.7)
    .. controls (-0.2, 0.4) .. (-0.1, 0.0);

  \draw[thick]
    (3.1, 3)
    .. controls (3.2, 2.6) .. (3.2, 2.3)
    -- (3.2, 0.7)
    .. controls (3.2, 0.4) .. (3.1, 0.0);

  \fill[blue!25]
    (0.0, 2.6) -- (2.6, 0.0) -- (3, 0.0) -- (3, 0.4) --
    (0.4, 3) -- (0.0, 3) -- cycle;

  \draw[<->, >=Stealth, thin]
    (0.0, 3.15) -- (0.4, 3.15)
    node[midway, above] {$\log(n)^{1+\delta}$};

 \draw[<->, >=Stealth, thin]
    (2.6, -0.15) -- (3, -0.15)
    node[midway, below] {$\log(n)^{1+\delta}$};

 \draw[<->, >=Stealth, thin]
    (3.3, 0) -- (3.3, 0.4)
    node[midway, right] {$\log(n)^{1+\delta}$};

  \draw[<->, >=Stealth, thin]
    (-0.3, 2.6) -- (-0.3, 3)
    node[midway, left] {$\log(n)^{1+\delta}$};

\node at (1.5, -1) {$(f)$ (Theorem \ref{theorembandmatrix})};

\end{tikzpicture}
    \end{minipage}    
    \begin{minipage}{0.45\textwidth}
    \centering
    \begin{tikzpicture}[scale=1.2]
    \useasboundingbox (-0.5, -1.2) rectangle (4.5, 3.8);
  \draw[thick]
    (-0.1, 3)
    .. controls (-0.2, 2.6) .. (-0.2, 2.3)
    -- (-0.2, 0.7)
    .. controls (-0.2, 0.4) .. (-0.1, 0.0);

  \draw[thick]
    (3.1, 3)
    .. controls (3.2, 2.6) .. (3.2, 2.3)
    -- (3.2, 0.7)
    .. controls (3.2, 0.4) .. (3.1, 0.0);

  \fill[gray!40!white]
    (0.0, 2.7) -- (2.7, 0.0) -- (3, 0.0) -- (3, 0.3) --
    (0.3, 3) -- (0.0, 3) -- cycle;

  \draw[<->, >=Stealth, thin]
    (0.0, 3.15) -- (0.3, 3.15)
    node[midway, above] {$(1+\delta)\log_p(n)$};

 \draw[<->, >=Stealth, thin]
    (2.7, -0.15) -- (3, -0.15)
    node[midway, below] {$(1+\delta)\log_p(n)$};

 \draw[<->, >=Stealth, thin]
    (3.3, 0) -- (3.3, 0.3)
    node[midway, right] {$(1+\delta)\log_p(n)$};

  \draw[<->, >=Stealth, thin]
    (-0.3, 2.7) -- (-0.3, 3)
    node[midway, left] {$(1+\delta)\log_p(n)$};

\node at (2.25, 2.25) {$0$};

\node at (0.75, 0.75) {$0$};

\node at (1.5, -1) {$(g)$ (\cite{M24})};

\end{tikzpicture}
    \end{minipage}
    \caption{The cokernels of $M(n)$ in (f) and $B(n)$ in (g) approach the Cohen--Lenstra distribution by Theorem \ref{theorembandmatrix} and \cite{M24}, respectively. Blue regions represent $\varepsilon$-balanced entries, gray regions represent Haar-uniform entries, and white regions represent entries with unrestricted distributions in (f) and entries fixed to zero in (g).}
    \label{fig3}
\end{figure}
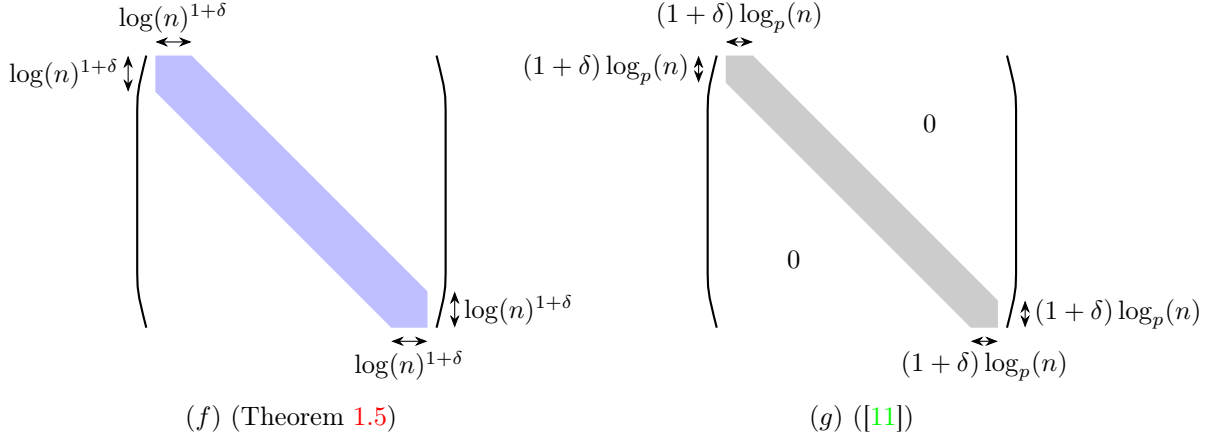

\begin{theorem}\label{theorembandmatrix}
    Let $p$ be a prime, $\delta,\varepsilon>0$ real numbers, and $B$ a finite abelian $p$-group. Let $M(n)$ be a random matrix valued in $\M_{n \times n}(\Z_p)$ with independent entries $M(n)_{ij}$, which are $\varepsilon$-balanced if ${|i-j| \leq \log(n)^{1+\delta}}$  and arbitrary otherwise. Then \[\lim_{n \to \infty} \P(\cok(M(n)) \simeq B) = \frac{\prod_{k=1}^\infty (1-p^{-k})}{|\Aut(B)|}.\]
\end{theorem}
Theorem \ref{theorembandmatrix} applies to a large variety of random matrix ensembles with some entries $\varepsilon$-balanced and others arbitrary, because we can freely permute rows and columns of $M(n)$. In particular, Theorem \ref{theorembandmatrix} implies many cases of Theorem \ref{mainthm}, such as all three examples in Figure \ref{fig1}. Theorem \ref{theorembandmatrix} also implies many cases of Singhal's \cite[Conjecture 4.2.2]{S26thesis} general conjecture  about the cokernels of random block matrix ensembles with some blocks having $\varepsilon$-balanced entries and other blocks being fixed to zero.

Kang, Lee, and Yu \cite{KLY24} have also studied random matrices over $\M_{n \times (n+u)}(\Z_p)$ with independent entries, each of which is either 0 or $\varepsilon$-balanced. Let $k \geq 1$, and for each $n$ choose $\alpha_n^{(1)},\ldots,\alpha_n^{(k)} > 0$ and $\beta_n^{(1)},\ldots,\beta_n^{(k)}>0$ such that $\lim_{n \to \infty} (n-\alpha_n^{(\ell)} - \beta_n^{(\ell)}) = \infty$ for all $1 \leq \ell \leq k$. Let $M(n)_{ij} = 0$ if there is some $\ell$ such that $i \leq \alpha_n^{(\ell)}$ and $j \leq \beta_n^{(\ell)}$, and let $M(n)_{ij}$ be $\varepsilon$-balanced otherwise. Then, Kang--Lee--Yu \cite[Theorem 10.3]{KLY24} prove that as $n \to \infty$ then $\cok(M(n))$ approaches the Cohen--Lenstra distribution. When $n-\alpha_n^{(\ell)} - \beta_n^{(\ell)} \geq \log(n)^{1+\delta}$, their result follows from Theorem \ref{theorembandmatrix}, after permuting rows by replacing $i$ by $(n+1-i)$. For further work classifying when matrices with Haar-uniform nonzero entries and zero entries in various patterns similar to \cite{KLY24,M24} will be Cohen--Lenstra universal, see \cite{JKLY26}. 

A natural question is to determine the minimum number of arbitrary $\varepsilon$-balanced entries in a matrix ensemble whose cokernel approaches the Cohen--Lenstra distribution. Let $Z_n \subseteq [n]^2$ for each $n \in \N$ and fix $\varepsilon>0$. Let $M_{Z}(n)$ be a random matrix valued in $\M_{n\times n}(\Z_p)$, such that $M_{Z}(n)_{ij}$ is $\varepsilon$-balanced if $(i,j) \in Z_n$, and $M_{Z}(n)_{ij}=0$  otherwise. Let $\omega(n)$ be the slowest growing function that there exists a sequence $\{Z_n\}_{n \geq 1}$ with $\omega(n) = |Z_n|$ where $\lim_{n \to \infty} \cok(M_{Z}(n))$ equals the Cohen--Lenstra distribution for any choice of the \mbox{$\varepsilon$-balanced} variables. Kang--Lee-Yu \cite[Remark 1.8]{KLY24} show that $\omega(n) < \left(\frac{1}{2}+\delta\right)n^2$ for any $\delta>0$, and ask \cite[Question 1.9]{KLY24} whether $\omega(n) < \frac{1}{2}n^2$. Theorem \ref{theorembandmatrix} proves that $\omega(n) \leq n\log(n)^{1+\delta}$, giving an affirmative answer to \cite[Question 1.9]{KLY24}.\footnote{While Theorem \ref{theorembandmatrix} proves $\omega(n) \leq 2n\log(n)^{1+\delta}$, we can change $\delta$ to get $\omega(n) \leq n\log(n)^{1+\delta}$ for any $\delta > 0$ and $n$ large.}

Combining this with the lower bound from \cite[Theorem 4.1]{KLY24} gives \[n\log_p(n) \leq \omega(n) \leq n\log(n)^{1+\delta}.\] Recently, Jung--Lee--Yu \cite{JLY26} found that Theorem \ref{thmWood} applies to random matrices $M(n)$ with $\varepsilon = C\frac{\log(n)}{n}$ for any $C>1$. So, there are random matrix ensembles valued in $\M_{n \times n}(\Z_p)$ where each matrix has an average of $C n \log(n)$ nonzero entries which appear in random locations, whose cokernels approach the Cohen--Lenstra distribution. Perhaps this heuristic suggests that ${\omega(n) \in \Theta(n \log n)}$. 

\subsection{Matrices with dependence between entries.}\label{sectdependent}
In general, many universality results in probability theory are more useful when versions are proved with weaker independence assumptions. (See e.g. \cite[Chapter 8]{Durrett19} for a treatment of the Central Limit Theorem with some dependence allowed.) So, a natural question, asked by Wood \cite[Open Problem 3.10]{Wood22} in her ICM talk, is to study which random matrix ensembles over $\Z$ or $\Z_p$ with dependent entries will have cokernels approaching the Cohen--Lenstra distribution.

There are many random matrix ensembles over $\Z$ or $\Z_p$ with dependent entries whose cokernels approach the Cohen--Lenstra distribution: (1) The Laplacian of an Erd\H{o}s--R\'enyi random directed graph \cite{NW22}; (2) The Laplacian of a random $d$-regular or bipartite directed graph $\Gamma$ \cite{M20,S26}; (3) A block matrix with dependence within blocks and independence of different blocks \cite{Gorokhovsky26}; (4) The ensemble ${M(n) = G-I_n}$ for $G$ a random matrix in $\GL_n(\Z_p)$ distributed according to Haar measure \cite{FW89}; (5) The ensemble ${M(2n) = G-I_{2n}}$ for $G$ a random matrix in a generalized symplectic group $\text{GSp}_{2n}^{(q)}(\Z_p)$ distributed according to Haar measure, where $q$ is a prime power such that $p \nmid q-1$, as $q \to \infty$ \cite{G15}. Determining which ensembles are in this universality class is a subtle problem, as illustrated by the many other random matrix ensembles whose cokernels approach different related distributions: symmetric matrices with independent entries on and above the main diagonal \cite{Wood14}, Laplacians of Erd\H{o}s--R\'enyi random undirected graphs \cite{Wood14}, Laplacians of random $d$-regular or bipartite undirected graphs \cite{M20,S26directed}, products of matrices with independent entries \cite{NVP24}, polynomials evaluated on matrices with independent entries \cite{CY23}, or matrices as in (5) when $p \mid q-1$ \cite[Corollary 3.2.7]{G15}. These examples illustrate that even a small amount of dependence can force the limiting cokernel distribution to differ from the Cohen--Lenstra distribution. For example, if $M(n)$ is a product of two random matrices with independent entries, then each entry in $M(n)$ mildly depends on those in its row and column and is independent of other entries, but Nguyen--Van Peski \cite{NVP24} showed that this dependence is strong enough to change the limiting distribution of $\cok(M(n))$.

We consider matrices where there can be dependence between entries within each column, but not between entries in different columns. Each column will have $n$ independent (possibly) random bits, but these bits will correspond to $n$ linear functions of the entries rather than exactly the $n$ entries in the column. We now define this model precisely.

\starpar{Let $p$ be a prime, $u \geq 0$ an integer, and $\varepsilon > 0$ a real number. Let $M(n)$ be a random matrix valued in $\M_{n \times (n+u)}(\Z_p)$ with independent columns $M(n)_1,\ldots,M(n)_{n+u}$. Let $A_1,\ldots,A_{n+u}$ in $\GL_n(\Z_p)$ be such that $\tilde{M}(n)_j \coloneq A_j M(n)_j$ has independent entries.
Let $e_1,\ldots,e_n$ be the canonical basis for $\Z_p^n$ and let $v_i^{j} = A_j^{-1}e_i$. 
Define the \emph{$\varepsilon$-balanced locus of $M(n)_j$} to be \[\hspace{-60 pt}W_j \coloneq \langle v_i^{j}: i \in [n] \text{ and } (\tilde{M}(n)_j)_i \text{ is $\varepsilon$-balanced} \rangle.\]}

We give three examples of columns $M(n)_j$ satisfying $(\star)$, which illustrate how $M(n)_j$ might have linear dependence between entries: (I) The top $n-1$ entries of $M(n)_j$ are selected independently from $\varepsilon$-balanced distributions, and the bottom entry is selected to make the entries in the columnn sum to 0; (II) The top $n-1$ entries of $M(n)_j$ are selected independently from $\varepsilon$-balanced distributions, and the bottom entry is selected to make the entries in the column excluding the top entry sum to 0; (III) Let $\lambda_1,\ldots,\lambda_n$ be independent $\varepsilon$-balanced variables, and let the $i^{th}$ entry of $M(n)_j$ be $X_i = \lambda_1 + \cdots + \lambda_i$. Using $(\star)$, these examples are represented by the following matrices $A_j$ which force $\tilde{M}(n)_j \coloneq A_jM(n)_j$ to have independent entries: 
\begin{equation*}
    A_j^{\textrm{(I)}} = \begin{pmatrix}
1 & 0 & 0 & \cdots & 0& 0 \\
0 & 1 & 0 & \cdots & 0& 0 \\
0 & 0 & 1 & \cdots & 0& 0 \\
\vdots & \vdots & \vdots & \ddots & \vdots & \vdots\\
0 & 0 & 0 & \cdots & 1 & 0 \\
1 & 1 & 1 & \cdots & 1& 1
\end{pmatrix} \text{, } A_j^{(\textrm{II})} = \begin{pmatrix}
1 & 0 & 0 & \cdots & 0& 0 \\
0 & 1 & 0 & \cdots & 0& 0 \\
0 & 0 & 1 & \cdots & 0& 0 \\
\vdots & \vdots & \vdots & \ddots & \vdots & \vdots\\
0 & 0 & 0 & \cdots & 1 & 0 \\
0 & 1 & 1 & \cdots & 1& 1
\end{pmatrix}\text{, } A_j^{(\textrm{III})} = \begin{pmatrix}
1 & 0 & 0 & \cdots & 0& 0 \\
-1 & 1 & 0 & \cdots & 0& 0 \\
0 & -1 & 1 & \cdots & 0& 0 \\
\vdots & \vdots & \vdots & \ddots & \vdots & \vdots\\
0 & 0 & 0 & \cdots & 1 & 0 \\
0 & 0 & 0 & \cdots & -1& 1
\end{pmatrix}.
\end{equation*}
In these three examples, all entries of $\tilde M(n)_j$ are $\varepsilon$-balanced, except $(\tilde M(n)_j)_n =0$ in (I) and (II). In (I), $W_j$ is all vectors in $\Z_p^n$ with column sum 0; in (II), $W_j$ is all vectors with column sum excluding the top entry being 0; in (III), $W_j = \Z_p^n$. The $\varepsilon$-balanced loci $W_j$ are necessary to state our theorem about matrices with dependence between the entries of each column.

\begin{theorem}\label{thmdependent}
    Assume the conditions of $(\star)$. Let $B$ be a finite abelian $p$-group and let $\alpha,\beta>0$ be real numbers such that $\alpha + \beta < 1$. Assume that at most $\alpha n$ entries of $\tilde M(n)_j$ are $\varepsilon$-degenerate for all $j \in [n+u]$. Assume that there exists a basis $u_1,\ldots,u_n$ for $\Z_p^n$ such that for all $i \in [n]$ we have \begin{equation}\label{eqweirdcondition}\#\{j \in [n+u]:u_i \not \in W_j\} \leq \beta n.\end{equation} Then \[\lim_{n \to \infty} \P(\cok(M(n)) \simeq B) = \frac{\prod_{k=1}^\infty (1-p^{-k-u})}{|B|^u\verts{\Aut(B)}}.\]
\end{theorem}
For example, the conditions of Theorem \ref{thmdependent} will be satisfied by any $M(n)$ with columns given by examples (I), (II), and (III), as long as there is a positive proportion of columns which are not as in (I) and a positive proportion of columns which are not as in (II). If all columns of $M(n)$ are as in (I) or all columns of $M(n)$ are as in (II), then $M(n)$ cannot have full rank because its rows will be linearly dependent. Another way to view this is that the columns of $M(n)$ will all be in the same $(n-1)$-dimensional submodule $W_1 = W_2 = \cdots  = W_{n+u}$. This explains the necessity of Theorem \ref{thmdependent} having a condition about the existence of a spanning set of $\Z_p^n$ whose elements are not in $W_j$ for too many indices $j$, as in (\ref{eqweirdcondition}).

We now see that Theorem \ref{mainthm} follows as a direct corollary of Theorem \ref{thmdependent}. Take $A_j = I_n$ to be the identity matrix for all $j \in [n+u]$, so that $M(n)_{j} = \tilde M(n)_j$ and $v^j_i = e_i$. Let $u_i = e_i$ for $i \in [n]$. Then, $u_i \in W_j$ if and only if $M(n)_{ij}$ is $\varepsilon$-balanced. So, the assumptions of Theorem \ref{thmdependent} reduce to $M(n)$ having at most $\alpha n$ $\varepsilon$-degenerate entries in each column and at most $\beta n$ $\varepsilon$-degenerate entries in each row. So Theorem \ref{thmdependent} implies Theorem \ref{mainthm}.

\subsection{Methods and Outline}
Our proofs use the moment method, as developed by Wood in \cite{Wood14} and \cite{Wood15}. In particular, we compute the $G$-moment $\E(\#\Sur(\cok(M(n)),G))$ for any finite abelian $p$-group $G$, in the limit as $n \to \infty$. Section \ref{sectmoments} states our results about these moments, and justifies that they imply Theorem \ref{thmdependent} and Theorem \ref{theorembandmatrix}. As in \cite{Wood15}, we compute these moments by replacing ${F \in \Sur(\cok(M(n)),G)}$ by $F \in \Sur((\Z/p^k\Z)^n, G)$ and estimating $\P(F(M) = 0) = \prod_{j = 1}^{n+u}\P(F(M_j) = 0)$. Thus, our proofs rely heavily upon the independence of different columns of $M$; it would be interesting to try to weaken this assumption using the methods in \cite{Wood14}. We end Section \ref{sectmoments} by fixing lots of notation.

In Section \ref{sectcolumns}, we will bound $\P(F(M_j) = 0)$ for $F \in \Sur(V,G)$ by proving inverse Littlewood-Offord theorems (Lemma~\ref{lemmacodes} and Lemma~\ref{lemmadepthprobabilitycolumns}). Our bounds will depend on how ``good'' $F$ is for the column $M_j$, which is quantified by whether $F$ is a \emph{code for $M_j$} and the \emph{depth of $F$ for $M_j$}. The definitions of code and depth for $M_j$ are the same as the definitions given in \cite{Wood15}, after replacing $M_j$ by $\tilde M_j =A_{j}M_j$ to make its entries independent, and removing all $\varepsilon$-degenerate entries of $M_j$. We also define the \emph{robust image of $F$ for $M_j$} to be a subgroup $H \leq G$ of index $D$ such that $F$ is robustly surjective onto $H$, and define whether it is \emph{proper}. Our proofs of Lemma~\ref{lemmacodes} and Lemma~\ref{lemmadepthprobabilitycolumns} are quite similar to the proofs of \cite[Lemma 2.1 and Lemma 2.7]{Wood15}. While \cite{Wood15} uses codes of distance $\delta n$ for some small $\delta>0$, we use codes of distance $\log(n)^{\lambda}$ for $\lambda >1$, which is necessary for the proof of Theorem \ref{theorembandmatrix}.

As in \cite{Wood15}, the main contributions to the moments come from codes. In Section \ref{sectoverview}, we describe a stratification of the $F$ which are not codes for all columns of $M$. We sort these $F$ by the number of columns for which they have each robust image, and the number of columns they have proper robust image for. Proposition~\ref{mainprop}(a) and (b) say that the $F$ in each category will have vanishing contributions to the moments, with $M$ distributed as in Theorem \ref{thmdependent} in (a) and Theorem \ref{theorembandmatrix} in (b). We prove that Theorem \ref{thmdependent} and Theorem~\ref{theorembandmatrix} follow from Proposition \ref{mainprop}(a) and (b), respectively.

In Section \ref{sectmainprop}, we prove Proposition \ref{mainprop}(a). This is done by bounding $\P(FM=0)$ for each category of $F$, and then bounding the number of $F$ in each category. Bounding the number of $F$ in each category is quite subtle; the main idea is that if $F$ has robust image $H$ for $M_j$, we ``almost'' have $F(W_j) \subseteq H$. By using these restrictions for many columns which have different values of $W_j$, we are able to bound the number of choices of $F(u_i)$ for each value of $u_i$, and thus bound the number of $F$. This work uses the full strength of the assumption in (\ref{eqweirdcondition}). The main challenges are that $M$ has a very general form so we do not know what each column looks like, and that each column $M_j$ of $M$ behaves well with respect to a specific basis $\{v_i^j\}$. 

In Section \ref{sectmainpropband}, we prove Proposition \ref{mainprop}(b). This follows a similar outline to Section \ref{sectmainprop}; again the technical lemma involves bounding the number of $F$ in each category. The main challenge is that $M$ may have very few $\varepsilon$-balanced entries, so each submodule $W_j$ may have very low dimension. So, we have to study all $n$ columns of $M$ at once, as opposed to in Proposition \ref{mainprop}(a), where the number of columns we use is independent of $n$. This makes bounding error terms very subtle, which is where we use the notion of proper robust image and the specific structure of band matrices. For example, Lemma \ref{lemma7}, a peculiar covering lemma about balls on the path graph $P_n$, is crucial for our proof. 

While the proof of Theorem \ref{thmdependent} is very general, the proof of Theorem \ref{theorembandmatrix} uses the structure of band matrices very heavily. This illustrates that the methods in Section \ref{sectcolumns} and Section \ref{sectoverview} can be applied to both general and specific random matrix ensembles. So, the main subtlety in showing that the cokernels of other random matrix ensembles will approach the Cohen--Lenstra distribution is to prove results similar to Proposition \ref{mainprop}, using similar methods to Section \ref{sectmainprop} or \ref{sectmainpropband}.

\section{The Moment Method}\label{sectmoments} We will prove Theorem \ref{thmdependent} and Theorem \ref{theorembandmatrix} using the moment method, which was developed by Wood in \cite{Wood14} and \cite{Wood15}. In this section, we state our key results about moments and use the machinery in \cite{Wood15} to show that they imply Theorem \ref{thmdependent} and Theorem \ref{theorembandmatrix}.

Let $M$ be a random matrix valued in $\M_{n \times (n+u)}(\Z/a\Z)$ for $a$ a power of $p$, formed by reducing each of the entries of $M(n)$ from Theorem \ref{thmdependent} from $\Z_p$ into $\Z/a\Z$. We now bound the $G$-moments $\E(\#\Sur(\cok(M),G))$ for any finite abelian $p$-group $G$. As we always have $a \cok(M) = 0$, these moments will only be nonzero when $G$ has exponent dividing $a$, i.e. $aG = 0$. 

\begin{theorem}\label{thmmoments}
Let $p$ be a prime and $a$ be a power of $p$. Follow the conditions of $(\star)$, except with all instances of $M(n)$, $\tilde{M}(n)$, and $\Z_p$ replaced by $M$, $\tilde{M}$, and $\Z/a\Z$, respectively. Let $G$ be a finite abelian group with exponent dividing $a$ and let $\alpha,\beta >0$ be real numbers such that $\alpha + \beta <1$. Assume that at most $\alpha n$ entries of $\tilde M_j$ are $\varepsilon$-degenerate for all $j \in [n+u]$. Assume that there exists a basis $\{u_1,\ldots,u_n\}$ for $(\Z/a\Z)^n$ such that for all $i \in [n]$ we have $\#\{j \in [n+u]:u_i \not \in W_j\} \leq \beta n$. Then, there exist $K,c > 0$, which are independent of $n$, such that
    \[\big|\E(\#\Sur(\cok(M), G)) -|G|^{-u}\big| \leq K\exp\left(-c\log (n)^2\right).\]
\end{theorem}

Now, let $M$ be a random matrix valued in $\M_{n \times n}(\Z/a\Z)$ for $a$ a power of $p$, formed by reducing each of the entries of $M(n)$ from Theorem \ref{theorembandmatrix} modulo $a$. We again bound the $G$-moments of $\cok(M)$.

\begin{theorem}\label{thmmomentsband}
Let $p$ be a prime and take $a$ to be a power of $p$. Let $\delta, \varepsilon > 0$ be real numbers and $G$ be a finite abelian group with exponent dividing $a$. Let $M$ be a random matrix valued in $\M_{n \times n}(\Z/a\Z)$ with independent entries such that $M_{ij}$ is $\varepsilon$-balanced if $|i-j| \leq \log(n)^{1+\delta}$ and arbitrary otherwise. Then, there exist $K,c > 0$, which are independent of $n$, such that
    \[\big|\E(\#\Sur(\cok(M), G)) -1\big| \leq K\exp\left(-c\log (n)^{1 + \delta/2}\right).\]
\end{theorem}

We now see that Theorem \ref{thmmoments} implies Theorem \ref{thmdependent} and Theorem \ref{thmmomentsband} implies Theorem \ref{theorembandmatrix}. Let $M(n)$ be the random matrix in Theorem \ref{thmdependent} or Theorem \ref{theorembandmatrix}, and let $M$ be formed by reducing its entries mod $a$, so that $M$ is as in Theorem \ref{thmmoments} or Theorem \ref{thmmomentsband}. For any finite abelian $p$-group $B$, this implies that ${\P(\cok(M) \simeq B) = \P(\cok(M(n))\otimes \Z/a\Z \simeq B)}$. Let $G$ be a finite abelian group with exponent dividing $a$. For any $p$-group $H$, surjections from $H$ to $G$ factor through $H \otimes \Z/a\Z$ so $\#\Sur(H,G) = \#\Sur(H \otimes \Z/a\Z, G)$. Therefore, $\E(\#\Sur(\cok(M), G)) = \E(\#\Sur(\cok(M(n)), G))$. 

So, Theorem~\ref{thmmoments} and Theorem \ref{thmmomentsband} ($u=0$ in this case) yield \[\lim_{n \to \infty}\E(\#\Sur(\cok(M(n)),G)) = \lim_{n \to \infty}\E(\#\Sur(\cok(M),G))=|G|^{-u}.\]  
Following \cite[Section 3]{Wood15}, these limiting $G$-moments of $\cok(M(n))$ determine the limiting distribution of $\cok(M(n))$ and force it to approach the Cohen--Lenstra distribution. This yields Theorem \ref{thmdependent} and Theorem~\ref{theorembandmatrix}. 

\begin{remark}
    It is possible to prove Theorem \ref{thmmoments} with a bound of $Ke^{-cn}$, by using the same logic as our proof with codes of distance $\delta n$ rather than $\log(n)^{2}$. We work with codes of distance $\log(n)^{\lambda}$ for $\lambda > 1$ throughout because this is necessary to prove Theorem \ref{thmmomentsband}.
\end{remark}

Let $M$ be as in Theorem \ref{thmmoments} or Theorem \ref{thmmomentsband}. Let $V = (\Z/a\Z)^n$ and $W = (\Z/a\Z)^{n+u}$, so $M \in \Hom(W,V)$. Surjections from $\cok(M) = V/MW$ to $G$ are equivalent to surjections $F: V \twoheadrightarrow G$ such that $F(MW) = 0$, or equivalently $FM = 0$, where $FM \in \Hom(W,G)$. Using the independence of the columns $M_1,\ldots,M_{n+u}$ of $M$ gives 
\begin{equation}\label{eq1}
    \E(\#\Sur(\cok(M), G)) = \sum_{F \in \Sur(V,G)} \P(FM=0) = \sum_{F \in \Sur(V,G)}\prod_{j=1}^{n+u} \P(F(M_j) = 0).
\end{equation} 

The rest of the paper is devoted to studying the right side of (\ref{eq1}) to prove Theorem \ref{thmmoments} and Theorem~\ref{thmmomentsband}.

\subsection{Notation and Conventions}
Let $p$ be a prime, $a$ be a power of $p$, $n>0$ be an integer, $u \geq 0$ be an integer, $\varepsilon>0$ and $\lambda >1$ be real numbers, and $G$ be a finite abelian group with exponent dividing $a$. Let $\{e_1,\ldots,e_n\}$ be the canonical basis for $V = (\Z/a\Z)^n$. When we work towards proving Theorem \ref{thmmoments}, let $M,\alpha,\beta,\tilde{M}_j,v_i^j,W_j,u_i$ be as in Theorem \ref{thmmoments} and take $\lambda =2$. When we work towards proving Theorem \ref{thmmomentsband}, let $\delta >0$ and $u =0$ and $M$ as in Theorem \ref{thmmomentsband} and take $\lambda =1 +\delta/2$. Whenever we refer to a random matrix $M$ valued in $\M_{n \times (n+u)}(\Z/a\Z)$, we are implicitly assuming that each of its columns $M_j$ is independent and $(\varepsilon,\tau_j,A_j)$-balanced (see Definition \ref{defepsbalanced}) and we are implicitly fixing $(\varepsilon,\tau_j,A_j)$ for each $j \in [n+u]$. This is necessary because any matrix $M$ may satisfy the conditions of either theorem with many values of $(\varepsilon,\tau_j,A_j)$ for each $j \in [n+u]$. Readers who are only interested in matrices with independent entries, i.e. Theorem \ref{mainthm} and Theorem \ref{theorembandmatrix}, should take $A_j = I_n$ to be the identity matrix throughout.

We now define some terms related to how we will do linear algebra over the ring $\Z/a\Z$. We will work with $(\Z/a\Z)$-submodules $W\subseteq V$. For elements $x_1,\ldots,x_\ell$ in $V$, let $\langle x_1,\ldots,x_\ell\rangle$ denote the submodule generated by $x_1,\ldots,x_\ell$, i.e. the set of $(\Z/a\Z)$-linear combinations of $x_1,\ldots,x_\ell$. For any set $\rho \subset [n]$, let $V_{\setminus \rho} = \langle e_i: i \not \in \rho\rangle$. For $V_1,V_2 \subseteq V$, let $V_1 + V_2 \subseteq V$ denote the module generated by the generators of $V_1$ and $V_2$. If $W \subseteq V$, define $\ol{W} \subseteq \F_p^n$ by reducing $W$ mod $p$. For $w \in W$, similarly define $\ol{w} \in \ol{W}$ by reducing $w$ mod $p$. We say that submodules or vectors are linearly independent if their reductions mod $p$ are linearly independent. Define the dimension $\dim W \coloneq \dim_{\F_p}\ol{W}$. 

We say that $W\subseteq V$ is a \emph{summand} if $V = W \oplus W'$ for some submodule $W'$. There are many equivalent definitions of summand, which we will use: 
\begin{itemize}
    \item $W \simeq (\Z/a\Z)^{r}$ for some $r \geq 0$.
    \item $W \simeq (\Z/a\Z)^{\dim W}$.
    \item $W = \langle w_1,\ldots,w_{\dim W}\rangle$ for some $w_1,\ldots,w_{\dim W}$ in $W$.
    \item $W\subseteq V$ has the same number of minimal generators as its reduction $\ol{W} \subseteq \F_p^n$.
    \item $W = \langle w_1,\ldots,w_r\rangle$ where $\{w_1,\ldots,w_r\}$ is a subset of a basis for $V$.
\end{itemize} If $W\subseteq V$ is a summand, then $\{w_1,\ldots,w_r\} \subseteq W$ is a basis of $W$ if and only if $r = \dim W$ and ${W = \langle w_1,\ldots,w_r\rangle}$. This is equivalent to $\{\ol{w_1},\ldots,\ol{w_r}\}$ forming an $\F_p$-basis for $\ol{W}$. Hence, if $W' \subseteq W$ are summands then any basis for $W'$ can be extended to a basis of $W$.

We say that $H \leq G$ if $H$ is a subgroup of $G$, and $H < G$ if $H$ is a proper subgroup of $G$. Let $[n]$ denote $\{1,\ldots,n\}$ and let $I_n$ be the $n \times n$ identity matrix. We use $\Hom(A,B)$ and $\Sur(A,B)$ to denote the homomorphisms and surjective homomorphisms, respectively, from $A$ to $B$. We sometimes write $\exp(x)$ for the exponential function $e^x$, and $\log$ with no base written is assumed to be $\ln$. We say $f(n) \in o(g(n))$ if $\lim_{n \to \infty} \frac{f(n)}{g(n)} = \infty$.

\section{Bounding \texorpdfstring{$\P(F(M_j) = 0)$}{P(FM = 0)}}\label{sectcolumns}

In this section, we will work to understand the factors on the right side of \eqref{eq1}, and partition the ${F\in \Sur(V,G)}$ based on how `good' they are for a column. This section generalizes many of Wood's definitions and lemmas in \cite[Section 2]{Wood15} to be useful in our setting, so we will often note which aspects of Wood's definitions and lemmas we are modifying. Throughout this section, fix $\lambda > 1$. We will take $\lambda = 2$ in our proof of Theorem \ref{thmmoments} and $\lambda = 1 +\delta/2$ in our proof of Theorem \ref{thmmomentsband}. 

We first define what it means for a random vector in $V$ to be $(\varepsilon,\tau,A)$-balanced, which allows for some entries of the vector, indexed by $\tau$, to not be $\varepsilon$-balanced and allows for dependence between entries if $A \neq I_n$. 

\begin{definition}\label{defepsbalanced} 
Let $X$ be a random vector valued in $V$ (with entries not necessarily independent). Let $\varepsilon > 0$ and $\tau \subset [n]$ and $A \in \GL_n(\Z/a\Z)$. We say that $X$ is \emph{$(\varepsilon,\tau, A)$-balanced} if $\tilde{X} = AX$ has independent entries and \[\tau \supseteq \{i: \tilde{X}_i \text{ is not }\varepsilon\text{-balanced} \}.\] 
\end{definition}

\begin{remark}
Following $(\star)$, such a column will have $\varepsilon$-balanced locus $W = A^{-1}V_{\setminus \tau}$. This submodule will appear in each definition in this section, which suggests that it is a natural submodule to consider. One should think of $W$ as a submodule of $V$ on which $X$ ``behaves well.''
\end{remark}

We now define what it means for a surjection $F$ to be a code for a random vector $X$ in $V$, when the entries of $X$ may not be independent or $\varepsilon$-balanced. We will later see that the vast majority of surjections $F$ are codes for $M$, i.e. codes for all columns of $M$.

\begin{definition}\label{defcode}
   Let $X$ be $(\varepsilon,\tau,A)$-balanced and fix $\lambda >1$. Then, we say that $F \in \Hom(V,G)$ is a \emph{code for $X$} if for all $\sigma \subset [n]$ with $|\sigma| < \log(n)^\lambda$, then $FA^{-1}(V_{\setminus\sigma \cup \tau}) = G$. For a random matrix $M$ valued in $\M_{n \times (n+u)}(\Z/a\Z)$, we say that $F$ is a \emph{code for $M$} if it is a code for all columns of $M$.
\end{definition}

Notice that this definition is the same as \cite[Definition 2]{Wood15} if $\tau = \emptyset$ and $A = I_n$. Definition \ref{defcode} becomes \cite[Definition 2]{Wood15} if we replace $FA^{-1}$ by $F$ and replace $V_{\setminus \tau}$ by $V$. The reason we use $FA^{-1}$ rather than $F$ is that $FX = FA^{-1}\tilde{X}$, and we want to work with $\tilde{X}$ because it has independent entries. We use $V_{\setminus \tau}$ rather than $V$ so that all entries of $\tilde X$ are $\varepsilon$-balanced. Note that in the usage of this definition throughout \cite{Wood15}, $\delta n$ with some small $\delta > 0$ replaces $\log(n)^\lambda$; we have to use codes of distance $\log(n)^\lambda$ rather than $\delta n$ in order to prove Theorem \ref{theorembandmatrix}. The idea of removing the $\varepsilon$-degenerate entries of $X$ to define codes also appears in \cite[Definition 7.6]{KLY24}, which applies when $\tau  = \{1,2,\ldots,m\} \subset [n]$ and $A = I_n$. 

We now generalize \cite[Lemma 2.1]{Wood15} to bound $\P(FX=g)$ when $X$ may have some $\varepsilon$-degenerate entries and dependence between entries, as long as $X$ is $(\varepsilon,\tau,A)$-balanced and $F$ is a code for $X$. 

\begin{lemma}\label{lemmacodes}
    Let $X$ be $(\varepsilon,\tau,A)$-balanced and let $F \in \Hom(V,G)$ be a code for $X$. Then, for all $g \in G$, \[\verts{\P(FX=g) - |G|^{-1}} \leq \exp(-\varepsilon\log(n)^\lambda/a^2).\]
\end{lemma}
We use the following estimate in the proof of Lemma \ref{lemmacodes}:
\begin{lemma}\cite[Lemma 4.2]{Wood14} \label{lemmacrude}
    Let $\zeta$ be a primitive $a^{th}$ root of unity. Let $y$ be an $\varepsilon$-balanced random variable valued in $\Z/a\Z$ and let $m$ be an integer such that $\zeta^m \neq 1$. Then $|\E(\zeta^{my})| \leq \exp(-\varepsilon/a^2)$.
\end{lemma}
\begin{proof}[Proof of Lemma \ref{lemmacodes}]
    We follow the proof from \cite[Lemma 2.1]{Wood15}. Let $\zeta$ be a primitive $a^{th}$ root of unity. Let $\tilde{X} = AX$ be $(\varepsilon,\tau,I_n)$-balanced. For $i \in [n]$, let $\tilde{X}_i$ denote the $i^{th} $ entry of $\tilde{X}$. Using the discrete Fourier transform and the independence of the entries $\tilde{X}_i$, \begin{align*}
        \P(FX=g) &= \P(FA^{-1}\tilde X=g) \\
        &= |G|^{-1}\sum_{C \in \Hom(G,\Z/a\Z)} \E(\zeta^{C(FA^{-1}\tilde{X}-g)}) \\ 
        &= |G|^{-1} + |G|^{-1} \sum_{C \in \Hom(G,\Z/a\Z) \setminus \{0\}} \zeta^{C(-g)} \prod_{i=1}^n\E(\zeta^{C(FA^{-1}(e_i))\tilde{X}_i}).
    \end{align*}
    For each $C \in \Hom(G,\Z/a\Z) \setminus \{0\}$, then $\ker(C) \neq G$. Since $F$ is a code for $X$, there must be at least $\log(n)^{\lambda}$ values of $i$ outside $\tau$ such that $FA^{-1}(e_i) \not \in \ker(C)$, i.e. $C(FA^{-1}(e_i)) \neq 0$. For these values of $i$, then $\tilde{X}_i$ is $\varepsilon$-balanced, so we can apply Lemma \ref{lemmacrude} with $m = C(FA^{-1}(e_i))$. For each nonzero $C$, this yields \[\verts{\prod_{i=1}^n\E(\zeta^{C(FA^{-1}(e_i))\tilde{X}_i})} \leq \exp(-\varepsilon \log(n)^{\lambda}/a^2).\] As there are exactly $|G|$ values of $C\in \Hom(G,\Z/a\Z)$, this implies the lemma.
\end{proof}

It is not sufficient to partition the $F$ into codes and non-codes for each column of $M$. So, we extend the definition of the depth \cite[Definition 3]{Wood15} of $F$ in $\Hom(V,G)$ to be with respect to a specific random $(\varepsilon,\tau,A)$-balanced vector $X$. 

\begin{definition}\label{defdepth}
    Let $X$ be $(\varepsilon,\tau,A)$-balanced, fix $\lambda > 1$, and take $F \in \Hom(V,G)$. The \mbox{\emph{depth of $F$ for $X$}} is the maximal integer $D$ such that there exists $\sigma \subset [n]$ with $|\sigma| < D\log(n)^{\lambda}$ and ${D = [G:FA^{-1}(V_{\setminus \sigma \cup \tau})]}$.
\end{definition}

If $\tau = \emptyset$ and $A = I_n$, this almost coincides with \cite[Definition 3]{Wood15}, except $|\sigma| < D\log(n)^{\lambda}$ replaces $|\sigma| < \log_p(D)\delta n$. We use $D$ rather than $\log_p(D)$ to ensure that robust images are unique; see Remark~\ref{robustimageunique}. We use $\log(n)^\lambda$ rather than $\delta n$ in order to prove Theorem \ref{theorembandmatrix}. Our definition is the same as \cite[Definition 3]{Wood15} if we replace $D\log(n)^{\lambda}$ by $\log_p(D)\delta n$, replace $FA^{-1}$ by $F$, and replace $V_{\setminus \tau}$ by $V$. 

\begin{remark}\label{rmkdepth1code}
    If $F$ is not a code for $X$, there exists $\sigma \subset [n]$ with $|\sigma| < \log(n)^{\lambda}$ such that ${FA^{-1}(V_{\setminus \sigma \cup \tau}) < G}$. So, the depth of $F$ for $X$ is at least $[G:FA^{-1}(V_{\setminus \sigma \cup \tau})] >1$. Therefore, if $F$ has depth 1 for $X$ then $F$ is a code for $X$.
\end{remark}

We now define the related notion of robust image, which refines depth by caring about the specific subgroup $FA^{-1}(V_{\setminus \sigma \cup \tau})$ of index $D$ in $G$. We also define of what it means to have proper robust image, which is a subtle definition that plays a crucial role in the proof of Theorem \ref{thmmomentsband}.

\begin{definition}\label{defrobustimage}
    Let $X$ be $(\varepsilon,\tau,A)$-balanced, and let $F$ in $\Hom(V,G)$ have depth $D$ for $X$. The \emph{robust image of $F$ for $X$} is a subgroup $H <G$ of index $D$ where there exists $\sigma \subset [n]$ with $|\sigma| < D\log(n)^{\lambda}$ satisfying $H = FA^{-1}(V_{\setminus \sigma \cup \tau})$. We say that $F$ has \emph{proper robust image} $H$ if, in addition, $H < FA^{-1}(V_{\setminus \tau})$.
\end{definition}

Note that $F$ always has a robust image for $X$ by Definition \ref{defdepth}. 

\begin{remark}\label{robustimageunique}
    We now see that robust images are unique. Let $X$ be $(\varepsilon,\tau,A)$-balanced, and let $F$ in $\Hom(V,G)$ have depth $D$ for $X$. Assume that $F$ has two distinct robust images $H_1$ and $H_2$ for $X$, so $[G:H_1] = [G:H_2] = D$. Let $k \in \{1,2\}$. Then there exists $\sigma_k \subset [n]$ with $|\sigma_k| < D\log(n)^{\lambda}$ such that $FA^{-1}(V_{\setminus \sigma_k \cup \tau}) = H_k$. Let $\sigma = \sigma_1 \cup \sigma_2$. Then $|\sigma| < 2D\log(n)^{\lambda}$ and $FA^{-1}(V_{\setminus \sigma \cup \tau}) \subseteq H_1 \cap H_2$. As $[G:H_1 \cap H_2] \geq 2D$, then $F$ has depth at least $2D$, which is a contradiction.
\end{remark}

We now generalize \cite[Lemma 2.7]{Wood15} to bound $\P(FX=0)$ where $X$ is $(\varepsilon,\tau,A)$-balanced and $F \in \Sur(V,G)$ has depth $D>1$ for $X$. If we take $A = I_n$ and $\tau = \emptyset$ to reproduce the conditions of \cite{Wood15}, Definition \ref{defrobustimage} implies that $F$ always has proper robust image for $X$. This is because $FA^{-1}(V_{\setminus \tau}) = F(V) = G$ and $F$ has robust image $H < G$ of index $D$.

\begin{lemma}\label{lemmadepthprobabilitycolumns}
    Let $X$ be $(\varepsilon,\tau,A)$-balanced. Let $F \in \Sur(V,G)$ have depth $D > 1$ for $X$. Then, \[\P(FX=0) \leq (D|G|^{-1} + \exp(-\varepsilon\log(n)^{\lambda}/a^2)) \cdot \begin{cases}
        (1-\varepsilon) &\text{if $F$ has proper robust image for $X$.} \\
        1 &\text{otherwise.} 
    \end{cases}\]
\end{lemma}

\begin{proof}
    Let $F$ have robust image $H$ for $X$, so $D = [G:H]$. Pick $\sigma \subset [n]$ with $|\sigma| < D\log(n)^{\lambda}$ such that $H = FA^{-1}(V_{\setminus\sigma \cup \tau})$. Let $v_i \coloneq A^{-1}(e_i)$.
Because $X = A^{-1} \tilde X$, \[FX = FA^{-1}\tilde X = \sum_{i=1}^n FA^{-1}(e_i)\tilde{X}_i =  \sum_{i \in \sigma \cup \tau} F(v_i)\tilde{X}_i + \sum_{i \not \in \sigma \cup \tau} F(v_i)\tilde X_i.\] By conditioning on the $\tilde X_i$ with $i \in \sigma \cup \tau$, \begin{align*}
        \P(FX=0) &=\P\left(\sum_{i \in \sigma \cup \tau} F(v_i)\tilde X_i \in H\right)\P\left(\sum_{i \not \in \sigma \cup \tau} F(v_i)\tilde X_i = -\sum_{i \in \sigma \cup \tau} F(v_i)\tilde X_i: \sum_{i \in \sigma \cup \tau} F(v_i)\tilde X_i \in H\right) \\
        &\leq  \P\left(\sum_{i \in \sigma \cup \tau} F(v_i)\tilde X_i \in H\right)\max_{g \in H}\left(\P\left(\sum_{i \not \in \sigma \cup \tau} F(v_i)\tilde X_i = g\right)\right).
    \end{align*}
    Let $\tilde X_{\setminus \sigma \cup \tau}$ denote the random vector in $V_{\setminus \sigma \cup \tau}$ produced by taking $\tilde X$ and removing the entries $\tilde X_i$ for $i \in \sigma \cup \tau$. Then $\tilde X_{\setminus \sigma \cup \tau}$ is $(\varepsilon,\emptyset,I_{n-|\sigma \cup \tau|})$-balanced. Assume for contradiction that $FA^{-1}|_{V_{\setminus \sigma \cup \tau}} \in \Hom(V_{\setminus \sigma \cup \tau},H)$ is not a code for $\tilde X_{\setminus \sigma \cup \tau}$. Then, we can remove a set of indices $\rho$ with $|\rho|<\log(n)^{\lambda}$ from $V_{\setminus \sigma \cup \tau}$ to get $FA^{-1}(V_{\setminus \rho \cup \sigma \cup \tau})<H$. Taking $\sigma' = \rho \cup \sigma$, then $|\sigma'| < (D + 1)\log(n)^{\lambda}$ and $[G:FA^{-1}(V_{\setminus \sigma' \cup \tau})]$ is a proper multiple of $D$. Then, $F$ has depth at least $[G:FA^{-1}(V_{\setminus \sigma' \cup \tau})] >D$ for $X$, which is a contradiction. 
    
    Hence, the restriction of $FA^{-1}$ to $V_{\setminus \sigma \cup \tau}$ is a code for $\tilde X_{\setminus \sigma \cup \tau}$. Applying Lemma \ref{lemmacodes} yields \[\P\left(\sum_{i \not \in \sigma \cup \tau} F(v_i)\tilde X_i = g\right) =\P\left(FA^{-1}(\tilde X_{\setminus \sigma \cup \tau})= g\right) \leq |H|^{-1} + \exp(-\varepsilon\log(n)^{\lambda}/a^2).\] 

    Now, let $F$ have proper robust image for $X$. Because $H=FA^{-1}V_{
    \setminus\sigma \cup \tau}$ and $H< FA^{-1}V_{\setminus\tau}$, there exists $k \in \sigma \setminus \tau$ such that $FA^{-1}(e_k) \not \in H$. As $k \not \in \tau$, then $\tilde X_k$ is $\varepsilon$-balanced. By conditioning on the $\tilde X_i$ for $i \in \sigma \cup \tau \setminus \{k\}$, then \[\P\left(\sum_{i \in \sigma \cup \tau} F(v_i)\tilde X_i \in H\right) = \P\left(F A^{-1}(e_k) \tilde X_k  \equiv -\sum_{i \in \sigma \cup \tau \setminus \{k\}} F(v_i)\tilde X_i \pmod{H}\right) \leq 1-\varepsilon.\]
    Multiplying the bounds from our previous two equations yields the lemma. 
\end{proof}

\section{Stratification of the surjections}\label{sectoverview}

In this section, we will prove Theorems \ref{theorembandmatrix} and \ref{thmdependent} assuming Proposition \ref{mainprop}. We will define our stratification of $F \in \Sur(V,G)$ which will allow us to state Proposition \ref{mainprop}. Assuming Proposition \ref{mainprop}, we will then prove Theorem \ref{thmmoments} and Theorem \ref{thmmomentsband}. By the work in Section \ref{sectmoments}, these yield Theorems \ref{thmdependent} and \ref{theorembandmatrix}, respectively. 

We begin with an essential definition. We will use all parts of Definition \ref{defstratify} repeatedly throughout the rest of the paper.

\begin{definition}[Stratification of $\Sur(V,G)$]\label{defstratify}
    Let $M$ be a random matrix valued in $\M_{n \times (n+u)}(\Z/a\Z)$. Let the subgroups of $G$ be $H_1,\ldots,H_s$, ordered so that $|H_k| \leq |H_{k+1}|$, which implies that $H_s = G$. Choose $n_1,\ldots,n_s$ to be nonnegative integers such that $\sum_{k=1}^s n_k = n+u$ and choose an integer $0 \leq r\leq n+u$. Define \begin{equation*}
        \Sur_{\{n_k\},r}(V,G) = \left\{F \in \Sur(V,G) : \begin{aligned}&\text{ For all $k \in [s]$, $F$ has robust image $H_k$ for exactly $n_k$ columns of $M$} \\ &\text{ and } F\text{ has proper robust image for }\text{exactly $r$ columns of $M$}\end{aligned}\right\}.\end{equation*}
\end{definition}

This stratification and the following proposition are motivated by the values of $n_1,\ldots,n_s$ and $r$ naturally showing up in our bound on $\P(FM = 0)$ in Lemma \ref{lemmaprobsimple}, which is proved using Lemma \ref{lemmacodes} and Lemma \ref{lemmadepthprobabilitycolumns}. We now state our main proposition, which says that the main contribution to the right side of (\ref{eq1}) comes from $F$ which are codes for $M$. 

\begin{proposition}\label{mainprop}
    Assume that either \begin{enumerate}[label = (\alph*)]
        \item $M$ is defined as in Theorem \ref{thmmoments}, with $\lambda = 2$, or
        \item $M$ is defined as in Theorem \ref{thmmomentsband}, with $\lambda = 1 + \delta/2$.
    \end{enumerate} Let $n_s < n+u$. Then, there exist constants $K,c >0$ which are independent of $n$ and $\{n_k\}$ and $r$ such that \begin{equation}\label{eqSur}\#\Sur_{\{n_k\},r}(V,G) \max_{F \in \Sur_{\{n_k\},r}(V,G)} \Big(\P(FM = 0)\Big) \leq Ke^{-c\log(n)^{\lambda}}.\end{equation}
\end{proposition}

For Proposition \ref{mainprop}(a), we could take any $\lambda > 1$, so we choose $\lambda = 2$ for concreteness. For Proposition~\ref{mainprop}(b), we need $1 < \lambda < 1 +\delta$ so we choose $\lambda = 1 + \delta/2$. 

Note that $n_s = n+u$ for $F$ is equivalent to $F$ having depth 1 for all columns of $M$, which implies that $F$ is a code for $M$ by Remark \ref{rmkdepth1code}. So, all $F \in \Sur(V,G)$ which are not codes for $M$ have $n_s <n+u$. Proposition~\ref{mainprop} says that the only nonvanishing contribution to the moments in (\ref{eq1}) comes from codes. We will prove Proposition \ref{mainprop}(a) in Section \ref{sectmainprop} and Proposition \ref{mainprop}(b) in Section \ref{sectmainpropband}. In each case, bounding $\P(FM = 0)$ follows from multiplying out the estimates in Section \ref{sectcolumns} but bounding $\#\Sur_{\{n_k\},r}(V,G)$ is quite subtle.

\begin{remark}
    For Proposition \ref{mainprop}(a), it would suffice to sort $F$ only by the number of columns for which it has each robust image, i.e. the data of $r$ is unnecessary. For Proposition \ref{mainprop}(a), we will also get a stronger bound of $Ke^{-cn}$ on the right side of (\ref{eqSur}).
\end{remark}

We now bound $\P(FM=0)$ when $F$ is a code for $M$. Our proof is identical to the proof of \cite[Lemma 2.4]{Wood15} if replace $\log(n)^\lambda$ by $\delta n$ for some $\delta >0$.

\begin{lemma}\label{lemmacodesmatrix}
Let $M$ be a random matrix valued in $\M_{n \times (n+u)}(\Z/a\Z)$. Fix $\lambda >1$ and let $F \in \Hom(V,G)$ be a code for $M$. Then, there exist $K,c > 0$ which are independent of $n$ and $F$ such that
    \[\verts{\P(FM=0) -|G|^{-n-u}} \leq \frac{Ke^{-c\log(n)^\lambda}}{|G|^{n+u}}.\]
\end{lemma}

To prove this, we use a simple inequality.
\begin{lemma}\cite[Lemma 2.3]{Wood15}\label{lemmasimple}
    Let $m \geq 2$ be an integer and $x \geq 0$ and $y$ be real numbers satisfying $|y|/x \leq 2^{1/(m-1)}-1$ and $x+y \geq 0$. Then, $|(x+y)^m-x^m| \leq 2mx^{m-1}|y|$.
\end{lemma}

\begin{proof}[Proof of Lemma \ref{lemmacodesmatrix}]
Using Lemma \ref{lemmacodes} and that $\lambda >1$, for all $j \in [n+u]$ and $n$ sufficiently large, we have \[|\P(FM_j = 0)-|G|^{-1}| \leq 
 \exp(-\varepsilon\log(n)^{\lambda}/a^2) \leq 
 \frac{\log 2}{|G|(n+u-1)} \leq
|G|^{-1} \left(2^{\frac{1}{n+u-1}}-1\right).\] For $n$ sufficiently large, we can thus apply Lemma \ref{lemmasimple} with $m = n+u$, $x = |G|^{-1},$ and $y = \P(FM_j = 0)-|G|^{-1}$ to get \begin{equation}\label{eq2}
    \left| \P(FM_j = 0)^{n+u}-|G|^{-n-u} \right|\leq 2(n+u)|G|^{-n-u+1}\left| \P(FM_j = 0)-|G|^{-1} \right|.
\end{equation}
Recall that $\P(FM=0) = \prod_{j=1}^{n+u} \P(FM_j = 0)$ by the independence of the columns of $M$ as in (\ref{eq1}). As (\ref{eq2}) holds for all $j \in [n+u]$, 
 \[\verts{\P(FM=0) -|G|^{-n-u}} \leq 2(n+u)\frac{\exp(-\varepsilon\log(n)^{\lambda}/a^2)}{|G|^{n+u-1}}.\] Taking $c < \varepsilon/a^2$ and $K=1$ yields the lemma for $n$ sufficiently large, and we can handle the smaller values of $n$ by choosing $K$ to be sufficiently large.
\end{proof} 

We also bound the number of $F$ which have depth greater than 1 for a column of $M$, using similar methods to those in the proof of \cite[Theorem 2.9]{Wood15}. By Remark \ref{rmkdepth1code}, all $F \in \Hom(V,G)$ which are not codes for $M$ will have depth greater than 1 for a column of $M$.

\begin{lemma}\label{lemmanoncodes}
    Let $M$ be as in Theorem \ref{thmmoments} with $n_0 = (1-\alpha)n$ or as in Theorem \ref{thmmomentsband} with $n_0 = \log(n)^{1+\delta}$. Then there exist constants $K,c > 0$ such that \[\#\{F \in \Hom(V,G): F \text{ has depth $>1$ for a column of $M$}\} \leq K|G|^{n} e^{-cn_0}.\]
\end{lemma}
\begin{proof}
    Let $F$ have depth $D>1$ for some column $M_j$. Let $F$ have robust image $H < G$ for $M_j$. Let $M_j$ be $(\varepsilon,\tau_j,A_j)$-balanced with $\ell \coloneq |[n] \setminus \tau_j|$. We can assume $\ell \geq (1-\alpha)n$ in the proof of Theorem \ref{thmmoments} by the assumptions of that theorem. We can assume $A_j = I_n$ and $\ell \geq \log(n)^{1+\delta}$ in the proof of Theorem \ref{thmmomentsband} because the assumptions of that theorem allow us to take all elements $i \in \tau_j$ to have $|i-j| > \log(n)^{1+\delta}$. So, we have $\ell \geq n_0$ in both cases. By Definition \ref{defrobustimage}, there exists $\sigma \subset [n]$ with $|\sigma| < D\log(n)^{\lambda} \leq |G| \log(n)^{\lambda}$ such that $FA_j^{-1}(V_{\setminus \sigma \cup \tau_j}) = H$. 

As $M$ has $n+u$ columns and $G$ has $s$ many subgroups, there are at most $(n+u)s$ choices for $j$ and $H$. Then, there are at most $\binom{\ell}{\floor{|G| \log(n)^{\lambda}}}$ choices of $\sigma$, if we assume $\sigma \subseteq [n] \setminus \tau_j$ and we enlarge $\sigma$ to have size $\floor{|G| \log(n)^{\lambda}}$ without loss of generality. Let $v_i \coloneq A_j^{-1}(e_i)$ so that $F(v_i) = FA_j^{-1}(e_i)$. Because $A_j$ is invertible, $\{v_1,\ldots,v_n\}$ is a basis for $V$ and $F$ is determined uniquely by $F(v_1),\ldots,F(v_n)$. There are $|G|$ choices for $F(v_i)$ for $i \in \sigma \cup \tau_j$, and $|H|$ choices for $F(v_i)$ for $i \not \in \sigma \cup \tau_j$. By the definitions of $\sigma$ and $\ell$, $|[n] \setminus (\sigma \cup \tau_j)|\geq \ell - |G| \log(n)^{\lambda}$. Clearly, $|H| \leq |G|/2$. Combining these estimates and letting $N \coloneq \#\{F \in \Hom(V,G): F \text{ has depth $>1$ for a column of $M$}\}$,
\begin{align}
    N &\leq (n+u)s \binom{ \ell}{\floor{|G| \log(n)^{\lambda}}} |G|^{|\sigma \cup \tau_j|} |H|^{|[n] \setminus (\sigma \cup \tau_j)|} \nonumber\\
     &\leq (n+u)s   \ell^{|G| \log(n)^{\lambda}} |G|^n \left(\frac{1}{2}\right)^{\ell - |G| \log(n)^{\lambda}}.\nonumber
\end{align}
Take $x_n \coloneq \frac{\ell}{\log(n)^{\lambda}}$. As $\ell \geq n_0$, we have $\log(x_n) \in o(x_n)$ and $\log(\log(n)) \in o(x_n)$. Taking $n$ sufficiently large and using these inequalities yields that, for some constant $c>0$,
\begin{align}N &\leq |G|^n(n+u)s  \exp\Bigg(\log(n)^{\lambda}\log(\ell)|G|-(\log 2)\big(\ell -  |G| \log(n)^{\lambda}\big)\Bigg)\nonumber \\
&= |G|^n(n+u)s  \exp\Bigg(\log(n)^{\lambda}\Big(|G|\big(\log(\log(n)^{\lambda}) + \log(x_n) +\log 2\big)-\log(2)x_n\Big)\Bigg)\nonumber \\
&\leq |G|^n\exp(-c\log(n)^{\lambda}x_n).\nonumber\end{align} Using that ${\log(n)^{\lambda}x_n = \ell \geq n_0}$ and taking $K$ to be large enough to account for all small values of $n$ yields the lemma. 
\end{proof}

Now, assuming Proposition \ref{mainprop}, we can prove Theorem~\ref{thmmoments} and Theorem~\ref{thmmomentsband}. These imply Theorem \ref{thmdependent} and Theorem \ref{theorembandmatrix} by the work in Section \ref{sectmoments}. The structure of this proof follows the proof of \cite[Theorem 2.9]{Wood15}, and the work has already been done for us by Proposition \ref{mainprop}, Lemma \ref{lemmacodesmatrix}, and Lemma \ref{lemmanoncodes}.

\begin{proof}[Proof of Theorem~\ref{thmmoments} and Theorem~\ref{thmmomentsband} assuming Proposition \ref{mainprop}]
    Let $M$ be as in Theorem~\ref{thmmoments} or Theorem~\ref{thmmomentsband}. Using (\ref{eq1}), we can expand
    \begin{align}
    \big|\E(\#\Sur(\cok(M), G)) -|G|^{-u}\big| &= \verts{\left(\sum_{F \in \Sur(V,G)} \P(FM = 0)\right)-|G|^{-u}}\nonumber \\
         &= \verts{\left(\sum_{F \in \Sur(V,G)} \P(FM = 0)\right)-\left(\sum_{F \in \Hom(V,G)} |G|^{-n-u}\right)}\nonumber \\
        & \leq \sum_{\substack{F \in \Sur(V,G) \\ F \text{ code for $M$}}} \verts{\P(FM = 0)-|G|^{-n-u}}\label{line3}  \\
        &\;\;\;\;\;\;+ \sum_{\substack{F \in \Sur(V,G) \\ F \text{ not code for $M$}}} \P(FM = 0)\label{line2} \\
        &\;\;\;\;\;\;+ \sum_{\substack{F \in \Hom(V,G) \\ F \text{ not code for $M$}}} |G|^{-n-u}.\label{line1}
    \end{align}
So, it suffices to show that each of \eqref{line3} and \eqref{line2} and \eqref{line1} is bounded above by $Ke^{-c\log(n)^\lambda}$ for some positive constants $K$ and $c$. Here, $\lambda = 2$ in the proof of Theorem \ref{thmmoments}, and $\lambda = 1 + \delta/2$ in the proof of Theorem \ref{thmmomentsband}.

For \eqref{line3}, this is true by Lemma \ref{lemmacodesmatrix}, since $\#\Sur(V,G) \leq |G|^n$. 

For \eqref{line2}, if $F$ is not a code for $M$, then $F$ has depth greater than $1$ for some column of $M$ by Remark~\ref{rmkdepth1code}. By Definition \ref{defstratify}, $n_s < n+u$. Using Proposition \ref{mainprop}(a) and (b), then
\begin{align*}
    \sum_{\substack{F \in \Sur(V,G) \\ F \text{ not code for $M$}}} \P(FM = 0) &\leq \sum_{\substack{n_1,\ldots,n_s \geq 0\\ n_s <n+u, \sum n_k = n+u}} \sum_{r \in [n+u]} \sum_{F \in \Sur_{\{n_k\},r}(V,G)} \P(FM=0) \\
    & \leq \sum_{\substack{n_1,\ldots,n_s \geq 0\\ n_s <n+u, \sum n_k = n+u}} \sum_{r \in [n+u]} \#\Sur_{\{n_k\},r}(V,G) \max_{F \in \Sur_{\{n_k\},r}(V,G)} \Big(\P(FM = 0)\Big) \\
    &\leq (n+u)^{s+1} Ke^{-c\log(n)^{\lambda}}.
\end{align*}
Since $(n+u)^{s+1}$ is polynomial in $n$ and $\lambda > 1$, we can make $c$ smaller and $K$ larger to yield the desired bound on (\ref{line2}).

For (\ref{line1}), we again note that if $F$ is not a code for $M$, then $F$ has depth greater than 1 for some column of $M$ by Remark~\ref{rmkdepth1code}. Applying Lemma \ref{lemmanoncodes} and using that $u \geq 0$ and $n_0 > \log(n)^\lambda$ in Lemma \ref{lemmanoncodes} yields the desired bound on (\ref{line1}).
\end{proof}

\section{Proof of Proposition \ref{mainprop}({\normalfont a})}\label{sectmainprop}

\subsection{Overview}\label{sect51}
In this section, we will prove Proposition \ref{mainprop}(a). Throughout this section, we freely assume the conventions of Definition \ref{defstratify}. When we refer to $M$, we are implicitly fixing $(\varepsilon,\tau_j,A_j)$ for each $j \in [n+u]$ such that $M_j$ is $(\varepsilon,\tau_j,A_j)$-balanced and $M$ satisfies the assumptions of Theorem \ref{thmmoments}. In the proof of Proposition \ref{mainprop}(a), we take $\lambda = 2$. However, in some lemmas we will simply let $\lambda > 1$, because we will wish to take $\lambda = 1 +\delta/2$ when we use these lemmas in Section \ref{sectmainpropband} in our proof of Proposition \ref{mainprop}(b).

We begin by defining two parameters $D_{1}$ and $D_2$, by letting $D_1$ be the maximal depth of $F$ for a column of $M$, and $D_2$ be the maximal number such that $F$ has depth $\geq D_2$ for many more than $\beta n$ columns of $M$. The definition of $D_2$ requires fixing a constant called $\theta$. The definitions of $D_1,D_2$, and $\theta$ will be used freely throughout this section.

\begin{definition}[Key Parameters $D_1,D_2,\theta$]\label{defkey}
    Choose $\theta$ such that $\beta < \theta < 1-\alpha$. Let $k_1 \geq 1$ be the minimum number such that $n_{k_1} > 0$. Let $D_1 \coloneq [G:H_{k_1}]$. Let $k_2$ be the minimum value such that $\sum_{k=1}^{k_2} n_k \geq \theta n$. Let $D_2 \coloneq [G:H_{k_2}]$. Since $k_1 \leq k_2$, we have $|H_{k_1}| \leq |H_{k_2}|$ and $D_1 \geq D_2$.
\end{definition}

In Lemma \ref{lemmaprobsimple}, we multiply our bounds for each column to bound $\P(FM=0)$ for all $F \in \Sur_{\{n_k\},r}(V,G)$ in terms of $\{n_k\}$ and $r$ or $D_1$ and $D_2$. In Lemma~\ref{mainlemma}, we will bound $\#\Sur_{\{n_k\},r}(V,G)$ in terms of $D_1$ and $D_2$. This is the main technical challenge of this section. To do this, we first show in Lemma \ref{lemmarobustimage} that the number of possibilities for $F|_{W}$ if $W \subseteq W_j$ is a summand and $F$ has depth $D$ for $M_j$ is approximately $(|G|/D)^{\dim W}$. We will then choose a column $M_{j_0}$ for which $F$ has depth $D_1$, and define the submodule $V_0 = W_{j_0} \subseteq V$ which has $\dim V_0 \geq (1-\alpha) n$. We bound the number of possibilities for $F|_{V_0}$ by Lemma~\ref{lemmarobustimage}. Then, we choose many columns of $M$ for which $F$ has depth $\geq D_2$ and use these columns to find linearly independent summands $V_1,\ldots,V_\ell \subseteq V$ where $V_k \subseteq W_{j_k}$ and $F$ has depth $\geq D_2$ for $M_{j_k}$. So, we can bound the number of possibilities for $F|_{V_k}$ using Lemma \ref{lemmarobustimage}. We then define a summand $V_{\ell+1}$ of small dimension such that $V_0+V_1+\cdots+V_\ell+V_{\ell+1} = V$. Multiplying our bounds on the number of possibilities for $F|_{V_0},\ldots,F|_{V_{\ell+1}}$ yields a bound on the number of such $F$. This yields Lemma \ref{mainlemma}, which is strong enough to prove Proposition~\ref{mainprop}(a) when $D_1>D_2$. 

To prove Proposition~\ref{mainprop}(a) when $D_1=D_2$, we must show that in all cases, we have a better bound on $\P(FM=0)$ or we can lower the bound on $\#\Sur_{\{n_k\},r}(V,G)$. We can lower $\P(FM=0)$ using Lemma \ref{lemmaprobsimple} except in the case where there are many columns of depth $D_1$ with different robust images. In this case, we can find two columns $M_j$ and $M_{j_0}$ with different robust images $H$ and $H_0$ of index $D_1$ and a large overlap $W_j \cap W_{j_0}$. So, $F(W_j \cap W_{j_0})$ almost sits inside $H \cap H_0$, which allows us to lower the number of possibilities for $F|_{W_j \cap W_{j_0}}$ and also for $F|_{W_{j_0}} = F|_{V_0}$.

\subsection{Bounding \texorpdfstring{$\P(FM=0)$}{P(FM=0)}}

We first bound the probability that $FM = 0$ for $F \in \Sur_{\{n_k\},r}(V,G)$ by multiplying out our bounds from Lemmas \ref{lemmacodes} and \ref{lemmadepthprobabilitycolumns}. To prove Theorem \ref{thmmoments}, we will mostly use the bound in (\ref{eqprobsimpleright}), except for one case where we use (\ref{eqprobsimpleleft}). To prove Theorem~\ref{thmmomentsband} in Section \ref{sectmainpropband}, we use (\ref{eqprobsimpleleft}).

\begin{lemma}\label{lemmaprobsimple}
Let $M$ be as in Theorem \ref{thmmoments} or Theorem \ref{thmmomentsband}. Let $F \in \Sur_{\{n_k\},r}(V,G)$ with $\lambda >1$. Then, there exists a constant $K >0$ which is independent of $n$ and $\{n_k\}$ and $r$ such that \begin{equation}\label{eqprobsimpleleft}
        \P(FM = 0) \leq K |H_1|^{-n_1}|H_2|^{-n_2}\cdots|H_{s}|^{-n_s}(1-\varepsilon)^r.
    \end{equation}
Furthermore, if $M$ is as in Theorem \ref{thmmoments}, define $D_1,D_2,\theta$ using Definition \ref{defkey}. Then 
\begin{equation}\label{eqprobsimpleright}
        \P(FM = 0) \leq K \left(\frac{D_1}{|G|}\right)^{\theta n}\left(\frac{D_2}{|G|}\right)^{(1-\theta)n}(1-\varepsilon)^r.
    \end{equation}
\end{lemma}

\begin{proof} Remark \ref{rmkdepth1code} says that for every column $M_j$ of $M$, then $F$ is either a code or has depth $D>1$. Lemma~\ref{lemmacodes} gives a bound on $\P(FM_j = 0)$ if $F$ is a code for $M_j$, and Lemma \ref{lemmadepthprobabilitycolumns} gives a bound on  $\P(FM_j = 0)$ if $F$ has depth $D >1$ for $M_j$. By Definition \ref{defstratify}, $F$ has robust image $H_k$ for $n_k$ columns and proper robust image for $r$ columns. Using the independence of the columns of $M$ and applying Lemma \ref{lemmacodes} or Lemma \ref{lemmadepthprobabilitycolumns} to each column, we get 
\[\P(FM=0) = \prod_{j=1}^{n+u} \P(FM_j = 0) \leq (1-\varepsilon)^r \prod_{k=1}^{s} \left(\frac{[G:H_k]}{|G|} + \exp\left(-\varepsilon \log(n)^{\lambda}/a^2\right)\right)^{n_k}.\] Multiplying and dividing by a constant to make the terms in the products equivalent, \begin{align*}\P(FM=0) &\leq (1-\varepsilon)^r \prod_{k=1}^{s} [G:H_k]^{n_k} \cdot \prod_{k=1}^{s} \left(\frac{1}{|G|} + \frac{\exp(-\varepsilon \log(n)^{\lambda}/a^2)}{[G:H_k]}\right)^{n_k}\\
&\leq (1-\varepsilon)^r \prod_{k=1}^{s} [G:H_k]^{n_k}\cdot  \left(|G|^{-1} + \exp(-\varepsilon \log(n)^{\lambda}/a^2)\right)^{n+u} \\
&\leq (1-\varepsilon)^r \prod_{k=1}^{s} [G:H_k]^{n_k}\cdot K|G|^{-n-u} \\
&\leq K (1-\varepsilon)^r \prod_{k=1}^{s} |H_k|^{-n_k}.
\end{align*}
Here, the third inequality came from the exact same logic as in the proof of Lemma \ref{lemmacodesmatrix}. This proves (\ref{eqprobsimpleleft}).

Now, let $M$ be as in Theorem \ref{thmmoments}. If $F$ has robust image $H$ for a column $M_j$, then $F$ has depth $D = [G:H]= |G|/|H|$ for that column. By Definition \ref{defkey}, $F$ has depth $\leq D_1$ for all columns of $M$ and depth $\leq D_2$ for at least $(n+u)-\theta n$ columns of $M$. So, all of the factors of the form $|H_k|^{-1}$ in (\ref{eqprobsimpleleft}) are bounded above by $D_1|G|^{-1}$, and at least $(n+u)-\theta n \geq (1-\theta)n$ of the $|H_k|^{-1}$ factors in (\ref{eqprobsimpleleft}) are bounded above by $D_2|G|^{-1}$. This proves (\ref{eqprobsimpleright}).
\end{proof}

\subsection{Bounding \texorpdfstring{$\#\Sur_{\{n_k\},r}(V,G)$}{the number of surjections}}
We now bound the number of elements in $\Sur_{\{n_k\},r}(V,G)$. This is the main technical lemma of Section \ref{sectmainprop}.

\begin{lemma}\label{mainlemma}
    Let $M$ be as in Theorem \ref{thmmoments}, let $\lambda = 2$, and define $D_1,D_2,\theta$ using Definition \ref{defkey}. Let $\zeta > 0$. Then, there exists $K>0$ such that  \begin{equation}\label{eqLemma53}
        \#\Sur_{\{n_k\},r}(V,G) \leq K \left(\frac{|G|}{D_1}\right)^{(1-\alpha)n} \left(\frac{|G|}{D_2}\right)^{\alpha n} e^{\zeta n}.
    \end{equation}
\end{lemma}

We begin by bounding the number of possibilities for $F|_{W}$, where $W$ is a summand which is a submodule of the $\varepsilon$-balanced locus $W_j$ of a column and $F$ has robust image $H$ for $M_j$. The idea is that there is a submodule $U \subseteq W_j$ of small codimension such that $F(U) \subseteq H$. This means that $F(U \cap W) \subseteq H$, so the number of possibilities for $F|_{W}$ are roughly $|H|^{\dim W}$, up to an error term which comes from the choice of $U$ and $F$ evaluated at the elements of $W \setminus U$. We use that $W$ is a summand to ensure that it can be generated by $\dim W$ many elements. This proof can be thought of as a basis-independent version of the proof of Lemma~\ref{lemmanoncodes}.

\begin{lemma}\label{lemmarobustimage}
    Let $\lambda >1$ and let $M_j$ be $(\varepsilon,\tau_j,A_j)$-balanced with $\tau_j \subsetneq [n]$ and $W_j \coloneq A_j^{-1}(V_{\setminus \tau_j})$. Let $W \subseteq W_j$ be a summand and let $H \leq G$ have $D = [G:H]$. Then \[
        \#\{\text{$F|_W$}: F \in \Sur(V,G) \text{ has robust image $H$ for $M_j$}\} \leq \Big(\left(n-|\tau_j|\right)|G|\Big)^{|G|\log(n)^{\lambda}} \left(\frac{|G|}{D} \right)^{\dim W}.
    \]
\end{lemma}
\begin{proof}
    If $n- |\tau_j| \leq |G|\log(n)^{\lambda}$, then $\dim W \leq \dim W_j  = n- |\tau_j| \leq |G|\log(n)^{\lambda}$. As $W$ is a summand, it is generated by $\dim W$ many generators. So there are at most $|G|^{\dim W} \leq |G|^{|G|\log(n)^{\lambda}}$ possibilities for $F|_W$, and the lemma holds. So, assume $n- |\tau_j| > |G|\log(n)^{\lambda}$.
    
    Let $F$ have robust image $H$ for $M_j$, with $[G:H] = D$. Let $\sigma \subseteq[n]$ with ${|\sigma| < D\log(n)^{\lambda} \leq |G|\log(n)^{\lambda}}$ satisfy $H = F(A^{-1}_jV_{\setminus \sigma \cup \tau_j})$, by Definition \ref{defrobustimage}. Define \[W' = W \cap A^{-1}_j V_{\setminus \sigma \cup \tau_j},\] so that $F(W') \subseteq H$. As $W \subseteq W_j$, \begin{equation}\label{eq5.1}
        W' =W \cap A^{-1}_j V_{\setminus \sigma}  \cap A^{-1}_j V_{\setminus \tau_j}= W \cap A^{-1}_j V_{\setminus \sigma} \cap W_j = W \cap A^{-1}_j V_{\setminus \sigma}.
    \end{equation} Without loss of generality, we can assume $\sigma \subseteq [n] \setminus \tau_j$ and we can shrink $W'$ by enlarging $\sigma$ to make ${|\sigma| = \floor{|G|\log(n)^{\lambda}}}$. This gives us at most $\binom{n - |\tau_j|}{\floor{|G|\log(n)^{\lambda}}}$ choices for $\sigma$, and hence $W'$. We can reduce (\ref{eq5.1}) mod $p$ and use that ${\dim_{\F_p} \big(\ol{A^{-1}_j V_{\setminus \sigma}}\big) = n-\floor{|G|\log(n)^{\lambda}}}$ to see that \begin{equation}\label{eqdimWprime}
        \dim W' =\dim_{\F_p} \left(\ol{W'}\right) \geq \dim_{\F_p} \left(\ol{W}\right) - \floor{|G|\log(n)^{\lambda}} = \dim W- \floor{|G|\log(n)^{\lambda}}.
    \end{equation}
    
    Choose $\{w_1,\ldots,w_{\dim W'}\} \subseteq W'$ deterministically so that $\{\ol{w_1},\ldots,\ol{w_{\dim W'}}\}$ is a basis for $\ol{W'}$. This implies $\langle w_1,\ldots,w_{\dim W'}\rangle$ is a summand  because it has the same number of generators as its reduction mod $p$. As $W$ is a summand, extend this deterministically to a basis $\{w_1,\ldots,w_{\dim W}\}$ for $W$. Then, $F|_W$ is determined uniquely by $F(w_1),\ldots,F(w_{\dim W})$. As $F(W') \subseteq H$, there are at most $|H| = |G|/D$ choices for $F(w_i)$ if ${1 \leq i \leq \dim W'}$. There are at most $|G|$ choices for $F(w_i)$ for $\dim W' < i \leq \dim W$. Multiplying our number of choices for $W'$ and number of choices for $F(w_1),\ldots,F(w_{\dim W})$ and using (\ref{eqdimWprime}) gives \begin{align*}\#\{F|_W: F \text{ has robust image $H$ for $M_j$}\} &\leq \binom{n - |\tau_j|}{\floor{|G|\log(n)^{\lambda}}} \left(\frac{|G|}{D} \right)^{\dim W'} |G|^{\dim W - \dim W'} \\
    &\leq \binom{n - |\tau_j|}{\floor{|G|\log(n)^{\lambda}}} \left(\frac{|G|}{D} \right)^{\dim W - \floor{|G|\log(n)^{\lambda}}} |G|^{\floor{|G|\log(n)^{\lambda}}}. 
    \end{align*} This is less than our desired bound because we always have $\binom{m}{k} \leq m^k$ for $m,k >0$.
\end{proof}

To prove Lemma \ref{mainlemma}, we will find many summands $V_0,V_1,V_2,\ldots,V_\ell,V_{\ell+1}$ which generate $V$ such that ${\dim V_0 \geq (1-\alpha)n}$ and $\dim(V_0+\cdots+V_\ell) \geq (1-\delta) n$, where $\delta >0$ can be arbitrarily small. $V_0$ will be a submodule of the $\varepsilon$-balanced locus of a column where $F$ has depth $D_1$, and $V_k$ for $1\leq k \leq \ell$ will be submodules of the $\varepsilon$-balanced loci of columns where $F$ has depth $\geq D_2$. By applying Lemma \ref{lemmarobustimage}, then there are at most $2^{\gamma n}(|G|/D_1)^{\dim V_0}$ possibilities for $F|_{V_0}$ and at most $2^{\gamma n} (|G|/D_2)^{\dim V_k}$ possibilities for $F|_{V_k}$ for $1\leq k\leq \ell$, where $\gamma>0$ is arbitrarily small and $n$ is sufficiently large. Multiplying these bounds out with the bound that there are at most $|G|^{\dim V_{\ell+1}}$ possibilities for $F|_{V_{\ell+1}}$ yields Lemma \ref{mainlemma}. 

\begin{proof}[Proof of Lemma \ref{mainlemma}]
    If $D_1 = 1$, then this lemma is trivial because $\#\Sur(V,G) \leq |G|^n$. So, let ${D_1 >1}$. Throughout this proof, we set $\lambda = 2$. We fix a random matrix ensemble $M \in \M_{n \times (n+u)}(\Z/a\Z)$ and utilize the conventions from Theorem \ref{thmmoments}. In particular, for $j \in [n+u]$ we let $M_j$ be $(\varepsilon,\tau_{j},A_{j})$-balanced with $|\tau_{j}| \leq \alpha n$ and $A_j \in \GL_n(\Z/a\Z)$. We take $v_i^j \coloneq A_j^{-1}(e_i)$ and $W_j = \langle v_i^j: i \in [n] \setminus \tau_j\rangle$.

    Let $F \in \Sur_{\{n_k\},r}(V,G)$. By Definition \ref{defkey}, choose $j_0 \in [n+u]$ such that $F$ has depth $D_1$ for the column $M_{j_0}$. Let $V_0 \coloneq W_{j_0}$, which is a summand since $W_{j_0}$ is generated by a subset of the basis $\{v_1^{j_0},\ldots,v_n^{j_0}\}$ for $V$. There are at most $n+u$ choices for $j_0$ and hence $V_0$, and at most $s$ choices for the robust image of $F$ for $M_{j_0}$. For $F \in \Sur_{\{n_k\},r}(V,G)$, Lemma \ref{lemmarobustimage} gives \begin{equation}\label{eqV0count}\#\{\text{possibilities for $V_0$ and $F|_{V_0}$}\} \leq (n+u)s \big(n|G|\big)^{|G|\log(n)^2} \left(\frac{|G|}{D_1}\right)^{\dim V_0}.
    \end{equation}

    Using the assumptions of Theorem \ref{thmmoments}, we now have two bases $\{v_1^{j_0},\ldots,v_n^{j_0}\}$ and $\{u_1,\ldots,u_n\}$ for $V$. Recall that $\{v_1,\ldots,v_n\}$ is a basis for $V$ exactly when $\{\ol{v_1},\ldots,\ol{v_n}\}$ is a basis for $\F_p^n$. By working over $\F_p$, this means that we can choose a set $S_1 \subseteq [n]$ such that $\{v_i^{j_0}:i \in [n]\setminus \tau_{j_0}\} \cup \{u_i:i\in S_1\}$ forms a basis for $V$. We fix one deterministic choice of $S_1$ for each $V_0$. As any basis for $V$ has $n$ vectors, then we can take $|S_1| = |\tau_{j_0}| \leq \alpha n$. Using this basis yields that \begin{equation}\label{eqlinind}
        V_0 + \langle u_i:i\in S_1 \rangle = V,
    \end{equation} with the two submodules on the left being linearly independent. 

    We now describe a sequence of linearly independent summands $V_1, V_2, V_3,\ldots, V_{\ell} \subseteq \langle u_i : i \in S_1\rangle$ ($\ell$ will be defined later), where $j_k \in [n+u]$ is such that $V_k \subseteq W_{j_k}$ and $F$ has depth $\geq D_2$ for $M_{j_k}$. In particular, we will describe a sequence of sets $S_1,S_2,S_3\ldots$ where $S_k \supseteq S_{k+1}$, and $V_k = \langle u_i: i \in S_k \setminus S_{k+1}\rangle$. Assuming that $S_k$ is defined, we will show how to define $S_{k+1}$ and $V_k$ in (\ref{eqdefVk}). Since $S_1$ is defined already, this will define $S_{k}$ and $V_k$ for all $1\leq k \leq \ell$.

    By Definition \ref{defkey}, there are at least $\theta n$ columns for which $F$ has depth at least $D_2$. By reordering the columns of $M$, which does not affect anything, call these columns $M_1,\ldots,M_{\ceil{\theta n}}$. Let $1 \leq j \leq \ceil{\theta n}.$ By Definition \ref{defepsbalanced}, $\tilde M_j \coloneq A_j M_j$ has independent entries and we can let $\tau_j = \{i: (\tilde M_j)_i \text{ is not $\varepsilon$-balanced}\}.$

    Using the assumptions of Theorem \ref{thmmoments}, we have  \[\#\{(i,j) \in S_k \times [\ceil{\theta n}]: u_i \not \in W_j\} =\sum_{i \in S_k}\#\{j \in [\ceil{\theta n}]:u_i \not \in W_j\} \leq\sum_{i \in S_k}\#\{j \in [n+u]:u_i \not \in W_j\} \leq |S_k|\beta n.\] As $\theta > \beta$ by Definition \ref{defkey}, this yields the following lower bound: \[\#\{(i,j) \in S_k \times [\ceil{\theta n}]: u_i \in W_j\} \geq |S_k| \ceil{\theta n} -|S_k|\beta n > 0.\] As these entries $(i,j)$ are split among $\ceil{\theta n}$ columns, there exists $j_k \in [\ceil{\theta n}]$ such that \begin{equation}\label{eqbound}\#\{i \in S_k: u_i \in W_{j_k}\} \geq \frac{|S_k|(\ceil{\theta n}-\beta n)}{\ceil{\theta n}} \geq |S_k|\left(1-\frac{\beta}{\theta}\right).\end{equation}
    
    Then, define \begin{equation}\label{eqdefVk}S_{k+1} \coloneq \{i \in S_k: u_i \not \in W_{j_k}\}\text{ and } V_k \coloneq \langle u_i : i \in S_k \text{ and } u_i \in W_{j_k} \rangle = \langle u_i: i \in S_k \setminus S_{k+1}\rangle.\end{equation} This clearly implies that $V_k \subseteq W_{j_k}$. 
    Because $V_k$ is generated by a subset of the basis $\{u_1,\ldots,u_n\}$, then $V_k$ is a summand. Let $F$ have robust image $H_{i_k}$ for $M_{j_k}$, where $[G:H_{i_k}] \geq D_2$. 
    
    Having chosen $j_1,\ldots,j_{k-1}$ and $i_1,\ldots,i_{k-1}$, we will now choose $j_k$ and $i_k$. There are at most $n+u$ choices for $j_k \in [n+u]$ and at most $s$ choices for $i_k \in [s]$. Choosing $j_k$ will inductively determine $S_{k+1}$ and $V_k$ by (\ref{eqdefVk}). For $F \in \Sur_{\{n_k\},r}(V,G)$ with $V_0,V_1,\ldots,V_{k-1}$ fixed, Lemma \ref{lemmarobustimage} gives \begin{equation}\label{eqVkcount}\#\{\text{possibilities for $V_k$ and $F|_{V_k}$}\} \leq (n+u)s \big(n|G|\big)^{|G|\log(n)^2} \left(\frac{|G|}{D_2}\right)^{\dim V_k}.
    \end{equation}

    Let $\ell \geq 1$. Let $V_{\ell+1} \coloneq \langle u_i: i \in S_{\ell+1}\rangle$, so (\ref{eqdefVk}) yields that $V_1 + \cdots + V_{\ell+1} =\langle u_i:i\in S_1 \rangle$ and these are all linearly independent. Then $V_0 + V_1 + \cdots + V_{\ell+1} = V$ by (\ref{eqlinind}), with all of these submodules being linearly independent summands. So, $F$ is determined uniquely by $F|_{V_0}, F|_{V_1},\ldots,F|_{V_\ell},F|_{V_{\ell+1}}$. Note that there are at most $|G|^{|S_{\ell+1}|} = |G|^{\dim V_{\ell+1}}$ possibilities for $F|_{V_{\ell+1}}$ because $F|_{V_{\ell+1}}$ is determined uniquely by $F(u_i)$ for $i \in S_{\ell+1}$. Multiplying this with (\ref{eqV0count}) and (\ref{eqVkcount}), 
    \begin{equation}\label{eqfirst}\#\Sur_{\{n_k\},r}(V,G) \leq ((n+u)s)^{\ell+1} \big(n|G|\big)^{(\ell+1)|G|\log(n)^2} \left(\frac{|G|}{D_1}\right)^{\dim V_0} \prod_{k=1}^\ell \left( \frac{|G|}{D_2}\right)^{\dim V_k} \cdot |G|^{\dim V_{\ell+1}}.
    \end{equation}
    For $n$ sufficiently large, there exists $C >0$ (which can depend on $G$) such that \[(n+u)s \big(n|G|\big)^{|G|\log(n)^2} \leq e^{C \log(n)^3}.\] Plugging this into (\ref{eqfirst}) and using that $n=\dim V_0 + \dim V_1 + \cdots + \dim V_\ell + \dim V_{\ell+1}$ by linear independence yields that, for $n$ sufficiently large,
    \begin{equation}\label{eqsecond}
        \#\Sur_{\{n_k\},r}(V,G)\leq e^{(\ell + 1)C\log(n)^3}|G|^n D_1^{-\dim V_0}  D_2^{-(n-\dim V_0 - \dim V_{\ell+1})}.
    \end{equation}

    By (\ref{eqbound}) and (\ref{eqdefVk}), $|S_{k+1}| \leq |S_k|\frac{\beta}{\theta}$. As $|S_1| \leq \alpha n$, this implies that $\dim V_{\ell + 1} = |S_{\ell+1}| \leq \alpha n \frac{\beta^\ell}{\theta^\ell}$. Because $D_1 \geq D_2$ and $\dim V_0 \geq (1-\alpha)n$, then $D_1^{-\dim V_0}  D_2^{-(n-\dim V_0)} \leq D_1^{(\alpha-1)n} D_2^{-\alpha n}$. This simplifies (\ref{eqsecond}) to
    \begin{equation}\label{eqfinal}\#\Sur_{\{n_k\},r}(V,G) \leq e^{(\ell + 1)C\log(n)^3}|G|^{n} \frac{1}{D_1^{(1-\alpha)n}}   \frac{1}{D_2^{\alpha n \left(1 - \left(\frac{\beta}{\theta}\right)^{\ell}\right)}}.
    \end{equation}
    We now choose $\ell$. Without loss of generality, let $\zeta < \alpha$. If $D_2 > 1$, choose $\ell \geq 1$ such that \[\ell (\log \beta - \log \theta) < \log \zeta -\log \alpha - \log(\log D_2),\] which is possible since $\log \beta - \log \theta < 0$ because $\theta > \beta$. Exponentiating both sides, $\alpha  \left(\frac{\beta}{\theta}\right)^{\ell}\log(D_2) < \zeta$. If $D_2 = 1$, choose $\ell =1$ and this will also hold. Then, for $n$ sufficiently large, because $\log(n)^3 \in o(n)$, we have \[(\ell+1)C \log(n)^3 + \alpha  \left(\frac{\beta}{\theta}\right)^{\ell}\log(D_2)n < \zeta n.\] Exponentiating both sides yields that, for $n$ sufficiently large, \begin{equation*}\label{eqgoal}e^{(\ell + 1)C\log(n)^3} D_2^{\alpha n \left(\frac{\beta}{\theta}\right)^{\ell}} \leq e^{\zeta n}.\end{equation*} Plugging this into (\ref{eqfinal}) yields Lemma \ref{mainlemma} for $n$ sufficiently large, and we can choose $K$ to be large enough so that (\ref{eqLemma53}) holds for all smaller values of $n$.
\end{proof}

\subsection{Putting it all together}

When $D_1 > D_2$, then Proposition \ref{mainprop}(a) follows quickly from Lemmas~\ref{lemmaprobsimple} and \ref{mainlemma}. If $D_1 = D_2$, we need to do some more work. We will split into three cases depending on $\{n_k\}$, which are roughly as follows: (1) $F$ has depth $<D_1$ for many columns; (2) $F$ has the same robust image for most columns; (3) $F$ has many different robust images of index $D_1$. Cases 1 and 2 follow by using more detailed bounds from Lemma \ref{lemmaprobsimple} but Case 3 is subtle, as we will improve the bound in (\ref{eqV0count}). We first find a reasonably large intersection $W = W_{j_0}\cap W_{j}$, where $F$ has distinct robust images $H_0$ and $H$ of index $D$ for $M_{j_0}$ and $M_{j}$, respectively. Using similar methods to the proof of Lemma \ref{lemmarobustimage} will allow us to find a submodule $W' \subseteq W$ of low codimension such that $F(W') \subseteq H_0 \cap H$. This will lower the number of possibilities for $F|_{W_{j_0}} = F|_{V_0}$, lowering the bound in (\ref{eqV0count}).
\begin{proof}[Proof of Proposition \ref{mainprop}(a)]
    Let $N \coloneq \#\Sur_{\{n_k\},r}(V,G)$ and $P \coloneq \displaystyle\max_{F \in \Sur_{\{n_k\},r}(V,G)} \Big(\P(FM = 0)\Big)$, so we want to bound $NP$. Let $q \coloneq \frac{D_1}{D_2}$. By Lemma \ref{lemmaprobsimple}, specifically (\ref{eqprobsimpleright}), \begin{equation}\label{eqP}
        P \leq K |G|^{-n} D_2^n q^{\theta n}.
    \end{equation} By Lemma \ref{mainlemma}, for any $\zeta >0$, \begin{equation}\label{eqN}
        N \leq K |G|^{n} D_2^{-n} q^{-(1-\alpha)n}e^{\zeta n}.
    \end{equation} Multiplying the last two inequalities, \begin{equation}\label{eqmain}
        NP \leq K \left(q^{\theta + \alpha -1}e^\zeta\right)^n.
    \end{equation} Note that $\theta +\alpha- 1 <0$ by Definition \ref{defkey}. If $D_1 >D_2$, then $q > 1$ and we can choose $0 < \zeta <(1 - \alpha - \theta)\log q$. Then $q^{\theta + \alpha -1}e^\zeta < 1$, so (\ref{eqmain}) gives $NP \leq Ke^{-cn}$ for some small $c>0$ and some $K>0$. This yields Proposition \ref{mainprop}(a) when $D_1 > D_2$. 
    
    If $D_1 = 1$, then $F$ has robust image $G$ for all columns of $M$ by Definition \ref{defrobustimage} and we thus have $n_s = n+u$, which contradicts the assumptions of Proposition \ref{mainprop}(a).

    So, we can let $D_1 = D_2 > 1$. Define $D \coloneq D_1 = D_2$. Then $q =1$ so (\ref{eqmain}) becomes $NP \leq Ke^{\zeta n}$. So, it suffices to multiply this bound on $NP$ by a factor of $e^{-\gamma n}$ for any fixed $\gamma > 0$. We will do this by splitting into three cases, using often that $\beta < \theta < 1-\alpha$.

    \textbf{Case 1:} $M$ has at most $(\theta + \alpha)n$ columns of depth $D$.

    Then, $M$ has at least $n(1-\theta-\alpha)>0$ columns of depth less than $D$, which means depth at most $D/2$. So, these columns have robust image of size at least $2|G|/D$, and all $n+u$ columns have robust image of size at least $|G|/D$. By Lemma \ref{lemmaprobsimple}, specifically (\ref{eqprobsimpleleft}), we have \[P \leq K \left(\frac{D}{|G|}\right)^{n(\theta+\alpha)}\left(\frac{D}{2|G|}\right)^{n(1-\theta-\alpha)+u} \leq K |G|^{-n}D^{n}2^{-n(1-\theta-\alpha)}.\] So we can multiply the bound on $P$ from (\ref{eqP}) by $2^{-n(1-\theta-\alpha)}$. Taking $\gamma = (1-\theta-\alpha) \log 2 > 0$ suffices, as we have multiplied the bound on $NP$ from (\ref{eqmain}) by $e^{-\gamma n}$.

    \textbf{Case 2:} There exists $k \in [s]$ such that $[G:H_k] = D$ and $n_k \geq \theta n$. 

    Let $F \in \Sur_{\{n_k,\},r}(V,G)$, so $F$ is a surjection. As $\{u_1,\ldots,u_n\}$ is a basis for $V$, there is some $i \in [n]$ such that $F(u_i) \not \in H_k$. By our assumptions in Theorem \ref{thmmoments}, there exist at most $\beta n$ columns $M_j$ such that $u_i \not \in W_j$. By the assumption that $n_k \geq \theta n$, there are at least $n(\theta - \beta)>0$ columns $M_j$ such that $u_i \in W_j$ and $F$ has robust image $H_k$ for $M_j$. If $u_i \in W_j = A_j^{-1}V_{\setminus \tau_j}$ and $F$ has robust image $H_k$ for $M_j$, Definition \ref{defrobustimage} gives that $F$ has proper robust image for $M_j$. So $F$ has proper robust image for at least $n(\theta - \beta)$ columns $M_j$. Therefore, $r \geq n(\theta - \beta)$. Then, Lemma \ref{lemmaprobsimple} allows us to multiply the bound on $P$ from (\ref{eqP}) by $(1-\varepsilon)^{n(\theta - \beta)}$. Taking $\gamma = -(\theta-\beta) \log(1-\varepsilon) >0$ suffices, as we have multiplied the bound on $NP$ from (\ref{eqmain}) by $e^{-\gamma n}$.
    
    \textbf{Case 3:} Case 1 and Case 2 do not hold. In this case, $F$ has depth $D$ for more than $(\theta + \alpha)n$ columns of $M$ and $F$ has multiple robust images of depth $D$ for different columns of $M$. 

    By the assumptions of Theorem \ref{thmmoments}, \[\#\{(i,j) \in [n] \times [n+u]: u_i \not \in W_j\} = \sum_{i=1}^n \#\{j \in [n+u]: u_i \not \in W_j\}\leq n(\beta n).\] Taking the remaining pairs $(i,j) \in [n] \times [n+u]$, we get \[\#\{(i,j) \in [n] \times [n+u]: u_i \in W_j\} \geq n(n+u)-\beta n^2.\] By permuting columns, using the assumptions of Case 3, let $F$ have depth $D$ for $M_j$ for all $j \in [\ceil{(\theta + \alpha)n}]$. Removing all the other $n+u-\ceil{(\theta + \alpha)n}$ columns, which each have $n$ entries, gives \[\#\{(i,j) \in [n] \times [\ceil{(\theta + \alpha)n}]: u_i \in W_j\} \geq n(n+u)-\beta n^2 - n(n+u-\ceil{(\theta + \alpha)n}) =n\ceil{(\theta + \alpha)n}-\beta n^2.\] By averaging over all $j \in  [\ceil{(\theta + \alpha)n}]$, there exists some $j_0 \in [\ceil{(\theta + \alpha)n}]$ where 
    \begin{equation}\label{eqj}
        \#\{i \in [n]: u_i \in W_{j_0}\} \geq \frac{n\ceil{(\theta + \alpha)n}-\beta n^2}{\ceil{(\theta + \alpha)n}} 
        \geq n - \frac{\beta n}{\theta + \alpha}.
    \end{equation}
    This replaces our choice of $j_0$ in the proof of Lemma \ref{mainlemma}, as in that lemma we chose $j_0$ to be arbitrary such that $F$ has depth $D$ for $M_{j_0}$. Let $F$ have robust image $H_0$ for $M_{j_0}$ with $[G:H_0] = D$. By the assumption of Case 3, we can choose $j \in [n+u]$ such that $F$ has robust image $H$ for $M_{j}$ and $[G:H] = D$ and $H_0 \neq H$. By the assumptions of Theorem \ref{thmmoments}, \begin{equation}\label{eqj'}
        \#\{i \in [n]: v_i^j \in W_{j}\} \geq (1-\alpha)n.
    \end{equation} 
    For notational convenience, we permute the columns of $M$ so that $j_0 = 1$ and $j=2$. This means that $W_{j_0}=V_0$ becomes $W_1$ and $W_j$ becomes $W_2$. We also take $G_1 \coloneq H_0$ and $G_2 \coloneq H$, so that $F$ has robust image $G_k$ for $M_k$ for $k \in \{1,2\}$.
    
    Adding the inequalities $\alpha(1-\theta - \alpha) >0$ and $\theta > \beta$ gives $(1-\alpha)(\theta + \alpha) > \beta$. Hence, $1-\alpha > \frac{\beta}{\theta+\alpha}$. Let $\phi \coloneq (1-\alpha) - \frac{\beta}{\theta+\alpha} > 0$. Then, by using that $\{u_1,\ldots,u_n\}$ and $\{v_1^2,\ldots,v_n^2\}$ are bases for $V$ and adding the dimension bounds which result from (\ref{eqj}) and (\ref{eqj'}), \begin{equation*}\label{eqsumming}
        \dim W_1 + \dim W_2 \geq n - \frac{\beta n}{\theta + \alpha}+(1-\alpha)n \geq n + \phi n. 
    \end{equation*}
    Let $W \coloneq W_1 \cap W_2$. As $\ol{W_1}$ and $\ol{W_2}$ are subspaces of $\F_p^n$ with $\dim_{\F_p}\ol{W_1} + \dim_{\F_p}\ol{W_2} \geq n + \phi n$, then \[\dim W = \dim (W_1 \cap W_2) = \dim_{\F_p}(\ol{W_1}\cap\ol{W_2}) \geq \phi n.\]
    
    Let $H' \coloneq G_1\cap G_2$. Since $G_1 \neq G_2$ have the same size, then $|H'| \leq \frac{|G_1|}{2} = \frac{|G|}{2D}$. We will now find a submodule $W' \subseteq W$ such that $F(W') \subseteq H'$. Specifically, let $k \in \{1,2\}$ and let $\sigma_k \subseteq [n]$ be defined as in Definition \ref{defrobustimage} so that $|\sigma_k| \leq |G|\log(n)^2$ and $F(A_k^{-1}(V_{\setminus \tau_k \cup\sigma_k})) \subseteq G_k$. Then, $F(W') \subseteq H'$ where we take \begin{align*}W' &\coloneq W \cap A_1^{-1}(V_{\setminus \tau_1 \cup\sigma_1}) \cap A_2^{-1}(V_{\setminus \tau_2 \cup \sigma_2})\\ &= W \cap W_{1} \cap A_1^{-1}(V_{\setminus \sigma_1}) \cap W_2  \cap A_2^{-1}(V_{\setminus \sigma_2}) \\&= W \cap A_1^{-1}(V_{\setminus \sigma_1}) \cap A_2^{-1}(V_{\setminus \sigma_2}).\end{align*}
    We can reduce this mod $p$ and use that ${\dim_{\F_p} \Big(\ol{A_1^{-1}(V_{\setminus \sigma_1}) \cap A_2^{-1}(V_{\setminus \sigma_2})}\Big) \geq n-2\floor{|G|\log(n)^{\lambda}}}$ to see that \begin{equation}\label{dimWprime2}
        \dim W' =\dim_{\F_p} \big(\ol{W'}\big) \geq \dim_{\F_p} \big(\ol{W}\big) - 2\floor{|G|\log(n)^{\lambda}} \geq \dim W- 2\floor{|G|\log(n)^{\lambda}}.
    \end{equation}

After $W'$ is fixed, choose $\{w_1,\ldots,w_{\dim W'}\} \subseteq W'$ deterministically so that $\{\ol{w_1},\ldots,\ol{w_{\dim W'}}\}$ is a basis for $\ol{W'}$. This implies $\langle w_1,\ldots,w_{\dim W'}\rangle$ is a summand. As $W_1$ is a summand, extend this deterministically to a basis $\{w_1,\ldots,w_{\dim W_1}\}$ for $W_1$. Define another summand $U = \langle w_{\dim W' +1},\ldots,w_{\dim W_1}\rangle$. Then $\langle w_1,\ldots,w_{\dim W'}\rangle+U = W_1$ and $\dim W' + \dim U = \dim W_1$.

     There are at most $(n+u)^2$ choices of $j_0$ and $j$, which determine $W_1$ and $W_2$ and $W$. There are at most $s^2$ choices of $G_1$ and $G_2$, which determine $H'$. There are at most $\binom{n}{\floor{|G|\log(n)^2}}^2 \leq n^{2|G|\log(n)^2}$ possibilities for $\sigma_1$ and $\sigma_2$, and hence $W'$. As $F(W') \subseteq H'$, there are at most $|H'| \leq |G|/2D$ possibilities for $F(w_i)$ for $1\leq i \leq \dim W'$. As $U \subseteq W_1$ is a summand, Lemma \ref{lemmarobustimage} gives that \[\#\{\text{possibilities for }F|_{U}\} \leq (n|G|)^{|G|\log(n)^2} \left(\frac{|G|}{D}\right)^{\dim U}.\] Multiplying all of these bounds out, using that $F|_{W_1}$ is determined by $F(w_1),\ldots,F(w_{\dim W'})$ and $F|_U$, and using (\ref{dimWprime2}) gives \begin{align*}\#\{\text{possibilities for }W_1\text{ and }F|_{W_1}\} &\leq (n+u)^2s^2(n^3|G|)^{|G|\log(n)^2} \left(\frac{|G|}{2D}\right)^{\dim W'} \left(\frac{|G|}{D}\right)^{\dim U} \\
     &\leq (n+u)^2s^2(n^3|G|)^{|G|\log(n)^2} \left(\frac{|G|}{D}\right)^{\dim W_1} \left(\frac{1}{2}\right)^{\phi n - 2|G|\log(n)^2}.\end{align*}
    As $W_1 = V_0$, this is the same as the bound in \eqref{eqV0count}, except for a factor of \begin{equation*}
        (n+u)s \big(2n\big)^{2|G|\log(n)^2} 2^{-\phi n} \leq 2^{-\phi n/2}, 
    \end{equation*}
    where this inequality holds for $n$ sufficiently large. Repeating the logic of Lemma~\ref{mainlemma}, we get the exact same bound as we achieved there, multiplied by $2^{-\phi n/2}$. With $\gamma = \phi\log(2)/2 > 0$, we have thus multiplied our bound on $N$ from (\ref{eqN}) by $e^{-\gamma n}$, and therefore also our bound on $NP$ from (\ref{eqmain}).
\end{proof}

\section{Proof of Proposition \ref{mainprop}({\normalfont b})}\label{sectmainpropband}

\subsection{Overview}
In this section, we will prove Proposition \ref{mainprop}(b). Throughout this section, we freely assume the conventions of Definition \ref{defstratify} and assume $M$ satisfies the assumptions of Theorem \ref{thmmomentsband}, as shown in Figure~\ref{fig3}(f). In particular, let $u=0$, let $\delta >0$ be a real number, and let $\lambda = 1 + \delta/2$. Take $w_n \coloneq \floor{\log(n)^{1+\delta}}$, and for all $j \in [n]$ let $M_j$ be $(\varepsilon,\tau_j,I_n)$-balanced with \begin{equation}\label{eqtj}
    \tau_j \coloneq \{i \in [n] : |i-j| > w_n\}.
\end{equation} We let $\{v_i\} = \{e_i\}$ be the standard basis for $V = (\Z/a\Z)^n$, as we only need to consider one basis for $V$ in this section. All submodules in this section will be generated by subsets of $\{v_1,\ldots,v_n\}$, so they are automatically summands. In our use of the definitions and lemmas from Section \ref{sectcolumns}, we will always take $A = I_n$.

Like in Section \ref{sectmainprop}, we will prove Proposition \ref{mainprop}(b) by bounding $\P(FM=0)$ and $\#\Sur_{\{n_k\},r}(V,G)$. We again use Lemma \ref{lemmaprobsimple} as our bound on $\P(FM=0)$.
The main technical challenge is to bound $\#\Sur_{\{n_k\},r}(V,G)$. Our analysis delicately depends on $r$, the number of columns where $F$ has proper robust image. Columns where $F$ has proper robust image will contribute a $(1-\varepsilon)$ factor to $\P(FM=0)$, but they will increase $\#\Sur_{\{n_k\},r}(V,G)$ because of the error terms in Lemma \ref{lemmarobustimage}. With careful analysis, we are able to prove the following bound on $\#\Sur_{\{n_k\},r}(V,G)$, which grows slowly enough in $r$ to prove Proposition~\ref{mainprop}(b).

\begin{lemma}\label{lemmabandcount}
    Let $M$ be as in Theorem \ref{thmmomentsband}. There exist constants $K,C >0$ which are independent of $n$ and $\{n_k\}$ and $r$ such that \[\#\Sur_{\{n_k\},r}(V,G) \leq K |H_1|^{n_1}|H_2|^{n_2} \cdots |H_s|^{n_s} \exp\left({\frac{Cr}{\log(n)^{\delta/4}}}\right).\]
\end{lemma}

To prove Lemma \ref{lemmabandcount}, we again use Lemma \ref{lemmarobustimage} to bound $\#\{\text{possibilities for $F|_W$}\}$ for $W \subseteq V_{\setminus \tau_j}$. However, in this case, $\dim V_{\setminus \tau_j} \leq 2w_n + 1 \in o(n)$. So, to bound the number of possibilities for $F=F|_V$, we will need to use many more columns $M_j$ in our analysis. For each column $M_j$ we use where $F$ has proper robust image, then Lemma \ref{lemmarobustimage} will give us an error term bounded by $\exp(\log(n)^{1+3\delta/4})$. To make sure that our overall error term is at most $\exp\left({\frac{Cr}{\log(n)^{\delta/4}}}\right)$, we must ensure that we only use $O(r/w_n)$ columns where $F$ has proper robust image. To do this, we study the geometry and structure of the band matrix $M$. 

We begin by proving a peculiar graph theory lemma, Lemma \ref{lemma7}, about the path graph $P_n$. The path graph is useful because the entries in $[n] \setminus \tau_j$ as in (\ref{eqtj}) can be viewed as a ball in $P_n$. Lemma \ref{lemma7} will imply that we only need to use Lemma \ref{lemmarobustimage} on $O(r/w_n)$ columns where $F$ has proper robust image. After this, we study the subgroup $F(V_{\setminus \tau_j}) \leq G$, which should be thought of as the (non-robust) image of $F$ for the column $M_j$. We prove Lemma \ref{lemmaregions}, which implies that there are at most $\exp\left(\frac{8rs}{\log(n)^{\delta/2}}\right)$ possibilities for $F(V_{\setminus \tau_j})$ for all $j \in [n]$. This is done by showing that columns where $F$ has certain images will appear in intervals inside $[n]$, and there cannot be too many intervals. So, there is a strong regularity to where the columns for which $F$ has each nonproper robust image may appear. 

Combining Lemma \ref{lemma7} and Lemma \ref{lemmaregions} allows us to prove Lemma \ref{lemmabandcount} by bounding the number of choices for $F(v_i)$ for all $i \in [n]$. To produce these bounds, we use all of the columns where $F$ has nonproper robust image, and apply Lemma \ref{lemmarobustimage} to $O(r/w_n)$ of the columns where $F$ has proper robust image. 

\subsection{A graph theoretic interlude}\label{sectgeometry}
We now study the indices in $[n]$ by viewing them as vertices of the path graph $P_n$, which has an edge between $i$ and $j$ when $|i-j|= 1$. Then, $|i-j|$ is the shortest-path distance from $i$ to $j$. We can reinterpret (\ref{eqtj}) in terms of balls as \begin{equation}\label{eqball}
    [n] \setminus \tau_j = B(j,w_n) \coloneq \{i \in [n]:|i-j| \leq w_n\}.
\end{equation} For $S \subseteq [n]$, define \begin{equation}\label{eqdefphi}B(S) \coloneq\bigcup_{j \in S} B(j,w_n)= \bigcup_{j \in S} ([n] \setminus \tau_j).\end{equation} If $S$ is a set of indices $j$ of columns of $M$, then $B(S)$ is the set of row indices $i \in [n]$ for which $v_i \in V_{\setminus \tau_j}$ for some $j \in S$. The following curious identity about the path graph will be central to the proof of Lemma~\ref{lemmabandcount}.

\begin{lemma}\label{lemma7}
    Let $R, S \subseteq [n]$, not both empty. Then there exists $Q \subseteq R$ such that \begin{equation}{\label{eqgraphtheory}}
        |B(Q \cup S)| \geq \min\{|R \cup S| + w_n/2,n\}\text{ and }|Q| \leq \begin{cases} 
        \frac{40|R|}{w_n} &\text{ if } S \neq \emptyset \\
            \max \left\{\frac{40|R|}{w_n},1\right\}&\text{ if } S = \emptyset.
        \end{cases} 
    \end{equation}
\end{lemma}

We have not attempted to optimize the constant 40. When $S \neq \emptyset$, the bound $|B(Q \cup S)| \geq |R \cup S|$ would be sufficient to prove Lemma \ref{lemmabandcount}. The stronger bound on $|B(Q \cup S)|$ becomes necessary when $S = \emptyset$ and $|R| \in o(w_n)$. The remainder of Section \ref{sectgeometry} will be devoted to proving Lemma \ref{lemma7}. 

For $a,b \in \Z$, define the interval \[[a,b] \coloneq 
        \{i \in [n]: a \leq i \leq b\}.\] For $j \in [n]$, it is easy to see that $B(j,w_n) = [j-w_n,j+w_n]$, so $w_n+1 \leq |B(j,w_n)| \leq 2w_n+1$.
    
    For $S \subseteq [n]$, let $m(S)$ be the minimal $m \in \Z_{\geq 0}$ such that we can write $S = [a_1,b_1]\cup \cdots \cup [a_m,b_m]$. When we write $S \neq [n]$ as a union of $m(S)$ intervals, this implies that all intervals are disjoint and $a_\ell-1 \not \in S$ and $b_\ell+1 \not \in S$ for all $1 \leq \ell \leq m$. We also assume without loss of generality that $b_\ell < a_{\ell+1}$ for all $1\leq \ell \leq m-1$. We now prove a lemma which relates $|B(S)|$ to $|S|$ and $m(B(S))$.

    \begin{lemma}\label{lemma5}
    If $S \subseteq [n]$ and $B(S) \subsetneq [n]$, then $|B(S)| \geq |S| + w_n m(B(S))$.
\end{lemma}
\begin{proof}
Let $B(S)$ be split into intervals \[B(S) = [a_1,b_1] \cup [a_2,b_2] \cup \cdots \cup [a_m,b_m]\subsetneq [n]\] with $m = m(B(S))$, so that $a_\ell - 1 \not \in B(S)$ and  $b_\ell + 1 \not \in B(S)$ for all $\ell \in [m]$. Define \[a_{\ell}' \coloneq \begin{cases}
    a_\ell + w_n &\text{ if $a_\ell > 1$}\\
    1 &\text{ if $a_\ell \leq 1$}
\end{cases} \text{ and }b_{\ell}' \coloneq \begin{cases}
    b_\ell - w_n &\text{ if $b_\ell < n$}\\
    n &\text{ if $b_\ell \geq n$}.
\end{cases}\]
Let $j \in S$, so $B(j,w_n) \subseteq B(S)$. As $B(j,w_n)$ is an interval, we must have $B(j,w_n) \subseteq [a_\ell,b_\ell]$ for some $\ell \in [m]$. By easy casework, this implies $j \in [a_\ell',b_\ell']$. So, we have \[S \subseteq [a_1',b_1'] \cup [a_2',b_2'] \cup \cdots \cup [a_m',b_m'].\] When $[a_\ell,b_\ell] \neq [n]$, then $\#[a_\ell',b_\ell'] \leq \#[a_\ell,b_\ell] - w_n$. As $B(S) \neq [n]$, this implies that $|S| \leq |B(S)| - w_n m$.
\end{proof}

We prove another lemma, which says that if $[a,b]$ is covered by a set of balls of the form $B(j,w_n)$, then it can be covered by a small number of these balls. This is proved using a greedy approach.

\begin{lemma}\label{lemma8}
    If $R \subseteq [n]$ and $[a,b] \subseteq B(R)$, then there exists $Q \subseteq R$ such that $[a,b] \subseteq B(Q)$ and \[|Q| \leq \frac{\#[a,b]}{w_n} +4.\]
\end{lemma}
\begin{proof}
Without loss of generality, let $a,b \in [n]$. We will use that $x \in B(R)$ if and only if $B(x,w_n) \cap R \neq \emptyset$.

Choose $j_1$ in $R$ to be an element of $B(a,w_n) \cap R$, which exists since $a \in B(R)$. Then ${a \in B(j_1,w_n)}$. If $j_1 \geq b-w_n$, we are done by taking $Q = \{j_1\}$. Now, assume that $j_\ell \in [a-w_n,b-w_n-1]$ for some $\ell \geq 1$. Then, $j_\ell + w_n+1 \in [a,b]$. So, we can let $j_{\ell + 1}$ be the maximal element of \begin{equation}\label{eqBR}B(j_\ell + w_n+1,w_n) \cap R = [j_\ell+1,j_\ell+2w_n+1]\cap R.\end{equation} We repeat this process until $j_q \not \in [a-w_n,b-w_n-1]$ for $q \geq 1$, which will eventually happen because $j_{\ell} < j_{\ell+1}$ always. As $j_{q-1} \in [a-w_n,b-w_n-1]$ and ${j_q \in [j_{q-1}+1,j_{q-1}+2w_n+1]}$, then $j_q \in [b-w_n,b+w_n]$. So $b \in B(j_q, w_n)$. As ${j_{\ell + 1} \in [j_\ell+1,j_\ell+2w_n+1]}$ for all $1 \leq \ell < q$, then $B(\{j_\ell,j_{\ell+1}\})$ is an interval. This implies that $B(\{j_1,\ldots,j_q\})$ is an interval which contains $a$ and $b$, so ${[a,b] \subseteq B(\{j_1,\ldots,j_q\})}$. Let $Q = \{j_1,\ldots,j_q\}$ so $|Q| = q$ and $[a,b] \subseteq B(Q)$. 

For all $1\leq \ell \leq q-2$, the maximality of $j_{\ell + 1}$ in (\ref{eqBR}) implies that \[j_{\ell+2} \in [j_{\ell+1} +1,j_{\ell+1}+2w_n+1]\setminus [j_\ell +1,j_\ell+2w_n+1]  \subseteq [j_\ell+2w_n+2,j_\ell+4w_n+2].\] So $j_{\ell+2}-j_{\ell} \geq 2w_n$. So, for all $\ell \geq 1$, if $2\ell+1 \leq q$ then $j_q-j_1 \geq j_{2\ell+1}-j_{1} \geq 2w_n\ell$.

As $j_1 \geq a-w_n$ and $j_q \leq b+w_n$, then $j_q - j_1 \leq \#[a,b] + 2w_n$. So, if $2\ell+1 \leq q$ then $2w_n\ell \leq \#[a,b] + 2w_n$. Let $\ell$ be maximal such that $2\ell+1 \leq q$, which implies that $2\ell +2 \geq q$. Combining this with the previous equation, \[w_nq \leq w_n(2\ell +2) \leq \#[a,b] + 4w_n,\] which yields the lemma.
\end{proof}

We are now ready to prove Lemma \ref{lemma7}. We will apply Lemma \ref{lemma5} to assume that $B(R) \setminus B(S)$ does not have too many intervals, and Lemma \ref{lemma8} to find $Q \subseteq R$ such that $B(Q)$ contains sufficiently many of these intervals.

\begin{proof}[Proof of Lemma \ref{lemma7}]
    Clearly, $|R \cup S| \leq |R|+|S|$. If $B(S) = [n]$, we are done by taking $Q = \emptyset$. So, let $B(S)\subsetneq [n]$. 

    Let $m = m(B(S))$. If $S \neq \emptyset$, then $m \geq 1$. Assume that $|R| \leq w_n m/2$. By Lemma \ref{lemma5}, we have \[|B(S)| \geq |S| + w_nm \geq |S|+|R| +w_n/2 \geq |R \cup S| + w_n/2.\] So we can choose $Q = \emptyset$. If $S = \emptyset$ and $|R| \leq w_n/2$, we can choose $Q = \{x\}$ for any $x \in R$, since $|B(Q)| \geq w_n+1 \geq |R \cup S| + w_n/2$. Henceforth, we assume that $|R| >w_n m/2$ and $|R| > w_n/2$.
        
    Let $B(S) = [a_1,b_1] \cup \cdots \cup [a_m,b_m]$ be written as a disjoint union of intervals. As $S \subseteq R \cup S$, then we can write \begin{equation*}
        B(R \cup S) = [a_1',b_1'] \cup \cdots \cup [a_m',b_m'] \cup [c_1,d_1] \cup \cdots \cup [c_k,d_k]
    \end{equation*} as a union of disjoint intervals where $[a_\ell,b_\ell] \subseteq [a_\ell',b_\ell']$. We can assume without loss of generality that $c_\ell -1 \not \in B(R \cup S)$ and $d_\ell + 1 \not \in B(R \cup S)$ for any $\ell \in [k]$. This implies that $\#[c_\ell,d_\ell] \geq w_n+1$, as $B(j,w_n) \subseteq [c_\ell,d_\ell]$ for some $j \in R \cup S$. As $[a_\ell',b_\ell'] \setminus [a_\ell,b_\ell]$ is the disjoint union of at most two intervals, we can rewrite \begin{equation}\label{eqsetminus}
        B(R \cup S) \setminus B(S) = [x_1,y_1]\cup \cdots \cup [x_N,y_N],
    \end{equation} as a disjoint union of intervals, where $N \leq 2m + k$, and $\#[x_\ell,y_\ell] \geq w_n+1$ for all $\ell >2m$. If we have ${|B(S)| \geq |R \cup S| + w_n/2}$, then taking $Q = \emptyset$ suffices to prove Lemma \ref{lemma7}. 
    
    So, let $|B(S)| < |R \cup S|+w_n/2$. Applying Lemma \ref{lemma5}, we get $|B(R \cup S)| \geq \min\{|R \cup S| + w_n/2,n\}$. So, using (\ref{eqsetminus}), define $t \in [N]$ and $y \in [x_t,y_t]$ uniquely such \begin{equation}\label{eqIZY}
        I_{t,y} \coloneq [x_1,y_1]\cup \cdots \cup [x_{t-1},y_{t-1}] \cup [x_t,y] \text{ satisfies } |I_{t,y}| = \min\{|R \cup S| + w_n/2,n\} - |B(S)|.
    \end{equation} Notice that $|B(S)| \geq |S|$ because $S \subseteq B(S)$. As $|R \cup S| \leq |R| + |S|$, this implies that $|I_{t,y}| \leq |R| +w_n/2$. As $\#[x_\ell,y_\ell] \geq w_n$ for $\ell > 2m$, this implies that $t \leq 2m + \ceil{\frac{|R|+w_n/2}{w_n}} \leq \frac{5|R|}{w_n}+2$.

    By the definition of $I_{t,y}$ and (\ref{eqsetminus}), then $I_{t,y} \subseteq B(R)$. We now apply Lemma \ref{lemma8} to each of the intervals in $I_{t,y}$ to find $Q \subseteq R$ such that $I_{t,y} \subseteq B(Q)$ with \begin{equation*}
        |Q| \leq \sum_{\ell=1}^{t-1} \left(4 + \frac{\#[x_\ell,y_\ell]}{w_n}\right)\nonumber + \left(4 + \frac{\#[x_t,y]}{w_n}\right) \leq 4t + \frac{|I_{t,y}|}{w_n} \leq 4\left(\frac{5|R|}{w_n}+2\right) + \frac{|R|+w_n/2}{w_n} \leq \frac{21|R|}{w_n} +9.
    \end{equation*} 
    As $|R| > w_n/2$, we have $|Q| \leq \frac{40|R|}{w_n}$. As $I_{t,y} \cap B(S) = \emptyset$, (\ref{eqIZY}) yields ${|I_{t,y} \cup B(S)| =  \min\{|R \cup S| + w_n/2,n\}}$. As $I_{t,y} \cup B(S) \subseteq B(Q \cup S)$, we are done.
\end{proof}

\subsection{Bounding Nonproper Robust Images.}
We begin with a simple definition, which is about the value of $F(V_{\setminus \tau_j})$. One should think of this as being about the image of $F$ for $M_j$, rather than the robust image.

\begin{definition}\label{defRk}
    Let $1 \leq k \leq s$. Define the region \[T_k \coloneq \{j \in [n]: F(V_{\setminus \tau_j}) \subseteq H_k\}.\]
\end{definition}

By Definition \ref{defrobustimage}, $T_k$ contains all $j$ such that $F$ has nonproper robust image $H_k$ for $M_j$. This implies that $|T_k| \geq n_k-r$. Note that the data of $T_1,\ldots,T_s$ is equivalent to the data of $F(V_{\setminus \tau_j})$ for all $j \in [n]$. It turns out that there are not that many possibilities for $T_k$, since we will see that $T_k$ splits into intervals which are each bookended by many columns with proper robust image.

\begin{lemma}\label{lemmaregions}
    Let $F \in \Sur_{\{n_k\},r}(V,G)$ and let $1 \leq k \leq s$. There are at most $n^{\frac{8r}{\log(n)^{\lambda}}}$ possibilities for $T_k$. 
\end{lemma}
\begin{proof}
    With $k = s$ then $T_k = [n]$ and this is trivial. So, assume $k < s$ and hence $H_k<G$ is a proper subgroup. As $F: V \to G$ is surjective, $T_k \subsetneq [n]$.

    Let $T_k \subseteq [n]$ be split into intervals $T_k = [a_1,b_1] \cup [a_2,b_2] \cup \cdots \cup [a_m,b_m]$ with $m = m(T_k)$, so that $a_{\ell} - 1 \not \in T_k$ and  $b_\ell + 1 \not \in T_k$ for all $\ell \in [m]$. Assume without loss of generality that $a_1 \geq 1$ and $b_m \leq n$ and $b_\ell \leq a_{\ell+1}$ for all $\ell \in [m-1]$. It suffices to show that $m \leq \frac{4r}{\log(n)^{\lambda}}$, since there are at most $n$ choices for each of $a_\ell$ and $b_\ell$.

    By (\ref{eqball}), \[V_{\setminus \tau_j} = \langle v_i : i \in [n] \setminus \tau_j\rangle = \langle v_i: i \in B(j,w_n)\rangle.\] So, $j \in T_k$ if and only if $F(v_i) \in H_k$ for all $i \in B(j,w_n)$. Let $b \coloneq b_\ell < n$ for some $\ell \in [m]$. As $b \in T_k$ and $b + 1\not \in T_k$, there exists $i \in (B(b+1,w_n) \setminus B(b,w_n))$ such that $F(v_i) \not \in H_k$. This set is either empty or equal to $\{b + 1 + w_n\}$, so therefore $b + 1 + w_n \leq n$ and $F(v_{b + 1 + w_n}) \not \in H_k$.

    Choose $c \in [b+1,b+\floor{\log(n)^{\lambda}}-1]$. Let \[\sigma_c \coloneq B(c,w_n) \setminus B(b,w_n) = [b+1+w_n,c+w_n],\] so $|\sigma_c|\leq \floor{\log(n)^{\lambda}}-1$. As $b \in T_k$, then $F(V_{\setminus \tau_b}) \subseteq H_k$. As $b+1+w_n \in B(c,w_n)$, then $F(V_{\setminus \tau_c}) \not \subseteq H_k$, so $c \not \in T_k$. Then \[F(V_{\setminus \sigma_c \cup \tau_c})\subseteq F(V_{\setminus \tau_b}) \subseteq H_k.\] Therefore, $F(V_{\setminus \sigma_c \cup \tau_c}) \subsetneq F(V_{\setminus \tau_c})$. By Definition \ref{defrobustimage} and the uniqueness of robust images from Remark \ref{robustimageunique}, this implies that $F$ has proper robust image for $M_c$, since its robust image will be contained in $F(V_{\setminus \sigma_c \cup \tau_c})$.

    So, for all $\ell \in [m]$ where $b_\ell <n$, then $F$ has proper robust image for $M_c$ for all $c \in [b_\ell +1, b_\ell+\floor{\log(n)^{\lambda}}-1]$, and $c \not \in T_k$. As $b + 1 + w_n \leq n$ and $\log(n)^\lambda < w_n$, each interval $[a_\ell,b_\ell]$ where $b_\ell < n$ is followed by \[\#[b_\ell +1, b_\ell+\floor{\log(n)^{\lambda}}-1] = \floor{\log(n)^{\lambda}}-1 > \frac{\log(n)^{\lambda}}{2}\] values $c \in [n]$ such that $F$ has proper robust image for $M_c$. There are at least $m-1$ intervals $[a_\ell,b_\ell]$ where $b_\ell < n$. If $m = 1$ and $b_1 = n$, then $a_1 > 1$ because $T_k \subsetneq [n]$. By similar logic, the interval $[a_1,b_1]$ is preceded by at least $\frac{\log(n)^{\lambda}}{2}$ values $c \in [n]$ such that $F$ has proper robust image for $M_c$.
    
    As there are a total of $r$ elements $c \in [n]$ where $F$ has proper robust image for $M_c$, this means that \[r \geq (\max\{m-1,1\})\frac{\log(n)^{\lambda}}{2} \geq \frac{m}{2} \cdot \frac{\log(n)^{\lambda}}{2}.\] Therefore $m \leq \frac{4r}{\log(n)^\lambda}$, and the lemma follows.
\end{proof}

\subsection{Bounding \texorpdfstring{$\#\Sur_{\{n_k\},r}(V,G)$}{the number of surjections}}
At last, we are ready to bound $\#\Sur_{\{n_k\},r}(V,G)$, using all of our work in this section until now. We first use Lemma \ref{lemmaregions} to fix $T_k$ for all $1 \leq k \leq s$. Then, Lemma~\ref{lemma7} allows us to use only a small fraction of the columns with nonproper robust images. Then, Claim \ref{claimlemma6} describes how we will apply Lemma \ref{lemmarobustimage}. Combining Claim \ref{claimlemma6} with our usage of Lemma \ref{lemma7} allows us to prove Lemma \ref{lemmabandcount}.

\begin{proof}[Proof of Lemma \ref{lemmabandcount}]
    We originally labeled the subgroups of $G$ as $H_1,H_2,\ldots,H_s$ arbitrarily such that ${|H_k| \leq |H_{k+1}|}$ for all $1\leq k\leq s-1$. We now modify this ordering, by requiring that if $|H_k| = |H_{k+1}|$, then $n_k \geq n_{k+1}$. In other words, for each $D \mid G$, we reorder all the subgroups $H_k$ of index $D$ so that they are labeled in decreasing order of $n_k$. This will be useful when $r < w_n$.
    
    If $n_s = n$, the lemma holds because $\#\Sur(V,G) \leq |G|^n$. So, we assume that $n_s < n$. Throughout the proof, we will often let $n$ be sufficiently large, as the constant $K$ can be made large enough so that the lemma holds for smaller values of $n$.

    By Lemma \ref{lemmaregions}, \[\#\{\text{possibilities for $T_1,\ldots,T_s$}\} \leq \left(n^{\frac{8r}{\log(n)^{\lambda}}}\right)^s = \exp\left(\frac{8rs \log(n)}{\log(n)^{\lambda}}\right) = \exp\left(\frac{8rs}{\log(n)^{\delta/2}}\right). \] So, throughout the proof, fix $T_1,\ldots,T_s$ by absorbing this factor into $C$.

    Let $1 \leq k \leq s$. Define \[R_k \coloneq \{j \in [n]: F\text{ has proper robust image $H_k$ for $M_j$}\}.\] By Definitions \ref{defrobustimage} and \ref{defRk}, $R_k \cup T_k$ contains all $j$ such that $F$ has robust image $H_k$ for $M_j$, so $|R_k \cup T_k| \geq n_k$.  Let $\hat{R}_k \coloneq R_1\cup\cdots\cup R_k$ and $\hat{T}_k \coloneq T_1\cup\cdots\cup T_k$. Then, ${|\hat{R}_k \cup \hat{T}_k| \geq n_1 + \cdots + n_k.}$ Definition \ref{defstratify} implies that $|\hat{R}_k| \leq r$.
    
    We now apply Lemma \ref{lemma7} repeatedly. Let $\kappa \geq 1$ be the minimum number such that $n_{\kappa} > 0$. (Notice that $\kappa$ is the same as $k_1$ from Definition \ref{defkey}.) This implies that $R_\kappa \cup T_\kappa \neq \emptyset$ and $R_k \cup T_k = \emptyset$ if $1 \leq k < \kappa$. Let $S_{\kappa-1} = \emptyset$. For $\kappa \leq k \leq s$, we apply Lemma~\ref{lemma7} with $R = \hat R_k$ and $S = S_{k-1} \cup T_k$, which are not both empty, to choose $Q = Q_k \subseteq \hat R_k$. Letting $S_k = S_{k-1} \cup T_k\cup Q_k$, Lemma~\ref{lemma7} gives that, for $\kappa \leq k\leq s$,
    \begin{equation}\label{eqlemma7}|B(S_k)| \geq \min \{|S_{k-1} \cup T_k \cup \hat R_k|+w_n/2,n\} \text{ and } |Q_k| \leq \begin{cases} 
        \frac{40|\hat R_k|}{w_n} &\text{ if } S_{k-1} \cup T_k \neq \emptyset \\
            \max \left\{\frac{40|\hat R_k|}{w_n},1\right\}&\text{ if } S_{k-1} \cup T_k = \emptyset.
        \end{cases}\end{equation}
    Let $\kappa \leq k \leq s$. As $\hat{R}_k \cup \hat{T}_k \subseteq S_{k-1} \cup T_k \cup \hat R_k$, we have that $|B(S_k)| \geq \min \{n_\kappa + \cdots + n_k+w_n/2,n\}$. For $\kappa-1 \leq k \leq s$, define ${V_k \coloneq \langle v_i : i \in B(S_k)\rangle}$. Then, $V_{\kappa - 1} = \{0\}$ and for $\kappa \leq k \leq s$, we have ${\dim V_k \geq \min \{n_\kappa + \cdots + n_k+w_n/2,n\}}$. In particular, $V_s = V$.
    
    We now show that we can apply Lemma \ref{lemmarobustimage} to bound the number of possibilities for $F|_{V_k}$ in terms of the number of possibilities for $F|_{V_{k-1}}$.
        
    \begin{claim}\label{claimlemma6}
        As above, let $\kappa \leq k \leq s$, let $S_k = S_{k-1} \cup T_k \cup Q_k$ where $Q_k \subseteq \hat R_k$, and let ${V_k \coloneq \langle v_i : i \in B(S_k)\rangle}$. Then, for $n$ sufficiently large, \[\#\{\text{possibilities for $F|_{V_k}\}$} \leq \#\{\text{possibilities for $F|_{V_{k-1}}$}\} |H_k|^{|B(S_k)|-|B(S_{k-1})|} \exp\left(|Q_k|\log(n)^{1 + 3\delta/4}\right).\]
    \end{claim}
    
    \begin{proof}[Proof of Claim \ref{claimlemma6}]
    By (\ref{eqdefphi}) and Definition \ref{defRk}, if $i \in B(T_k)$ then $i \in B(j,w_n)$ for some $j \in T_k$, so ${F(v_i) \in H_k}$. So there are at most $|H_k|$ choices for $F(v_i)$. Let $V_k^0 \coloneq \langle v_i: i \in B(S_{k-1} \cup T_k)\rangle$. Clearly, $B(S_{k-1}) \subseteq B(S_{k-1}\cup T_k)$. This gives that \begin{equation}\label{eq101}
        \#\{\text{possibilities for $F|_{V_k^0}\}$} \leq \#\{\text{possibilities for $F|_{V_{k-1}}$}\} |H_k|^{|B(S_{k-1}\cup T_k)|-|B(S_{k-1})|}.
    \end{equation} Now, let $Q_k = \{j_1,\ldots,j_q\}$, with $q = |Q_k|$. For $1 \leq \ell \leq q$, we define \[V_k^\ell \coloneq \langle v_i:i\in B(S_{k-1}\cup T_k \cup \{j_1,\ldots,j_{\ell}\})\rangle.\] 
    By definition, $V_k = V_k^q$. Define \[U_k^\ell \coloneq \left\langle v_i: i \in \big(B(j_\ell,w_n) \setminus B(S_{k-1}\cup T_k \cup \{j_1,\ldots,j_{\ell-1}\})\big)\right\rangle \subseteq V_{\setminus \tau_{j_\ell}}.\] 
    Then $V_k^\ell =  V_k^{\ell-1} + U_k^\ell$ and $\dim V_k^\ell = \dim V_k^{\ell-1}+\dim U_k^\ell.$ As $Q_k \subseteq \hat R_k$, let $j_\ell \in R_{k'}$ for $1\leq k' \leq k$, so that $F$ has robust image $H_{k'}$ for $M_{j_\ell}$. There are at most $n$ choices of $j_\ell$. As $k' \leq k$, then $|H_{k'}| \leq |H_k|$ and there are at most $s$ choices of $k'$. By applying Lemma \ref{lemmarobustimage} as $U_k^\ell\subseteq V_{\setminus \tau_{j_\ell}}$ and letting $n$ be sufficiently large, \begin{equation}\label{eq102}
        \#\{\text{possibilities for $F|_{U_k^\ell}$}\}\leq ns\big((n-|\tau_{j_\ell}|)|G|\big)^{|G|\log(n)^{\lambda}}|H_{k'}|^{\dim U_k^\ell} \leq \exp\left(\log(n)^{1+3\delta/4}\right)|H_k|^{\dim V_k^\ell-\dim V_k^{\ell-1}}.
    \end{equation} The second inequality above holds only for $n$ sufficiently large, and uses that $\lambda = 1+\delta/2 < 1 + 3\delta/4$ and $\log(n-|\tau_{j_\ell}|) \leq \log(2w_n+1)$ is less than any positive power of $\log(n)$. As $V_k^\ell =  V_k^{\ell-1} + U_k^\ell$, then $F|_{V_k}$ is determined by $F|_{V_k^0},F|_{U_k^1},\ldots,F|_{U_k^q}$. We now multiply out (\ref{eq101}) and (\ref{eq102}) over all $1 \leq \ell \leq q$ and use that ${\dim V_k^0 = |B(S_{k-1}\cup T_k)|}$ to get \[\#\{\text{possibilities for $F|_{V_k}\}$} \leq \#\{\text{possibilities for $F|_{V_{k-1}}$}\} |H_k|^{\dim V_k - |B(S_{k-1})|} \left(\exp(\log(n)^{1+3\delta/4})\right)^q.\] As $\dim V_k = |B(S_k)|$ and $q = |Q_k|$, this proves Claim \ref{claimlemma6}.
    \end{proof}

    Let $Q \coloneq \sum_{k=\kappa}^s|Q_k|$. As $V_s = V$ and $V_{\kappa-1} = \{0\}$, multiplying out Claim \ref{claimlemma6} for all $\kappa \leq k \leq s$ yields \begin{equation}\label{eqtwoterms}\#\Sur_{\{n_k\},r}(V,G) \leq  \exp\left(Q\log(n)^{1 + 3\delta/4}\right) \cdot \prod_{k = \kappa}^s \left(|H_k|^{|B(S_k)|-|B(S_{k-1})|}\right).\end{equation}
    Let $s' \leq s$ be the minimum value such that $n_\kappa + \cdots + n_{s'}+w_n/2 \geq n$, so $s' \geq \kappa$. We now use that $|B(S_k)| \geq \min \{n_\kappa + \cdots + n_k+w_n/2,n\}$ and $|H_k| \leq |H_{k+1}|$ for all $\kappa \leq k \leq s$ to simplify (\ref{eqtwoterms}): \begin{align}
        \#\Sur_{\{n_k\},r}(V,G) &\leq \exp\left(Q\log(n)^{1 + 3\delta/4}\right) \cdot|H_s|^{|B(S_s)|}\prod_{k = \kappa}^{s-1} \left(\frac{|H_k|}{|H_{k+1}|}\right)^{|B(S_k)|}\nonumber\\
        &\leq \exp\left(Q\log(n)^{1 + 3\delta/4}\right) \cdot|G|^n \prod_{k = \kappa}^{s-1} \left(\frac{|H_k|}{|H_{k+1}|}\right)^{\min\{n_\kappa + \cdots + n_k+w_n/2,n\}}\nonumber \\
        &= \exp\left(Q\log(n)^{1 + 3\delta/4}\right) \cdot|H_{s'}|^{n-(n_\kappa + \cdots + n_{{s'}-1}+w_n/2)}|H_{\kappa}|^{n_\kappa +w_n/2}\prod_{k = \kappa+1}^{{s'}-1} |H_k|^{n_k}\footnotemark\nonumber \\
        &\leq \exp\left(Q\log(n)^{1 + 3\delta/4}\right) \cdot\left(\frac{|H_{\kappa}|}{|H_{s'}|}\right)^{w_n/2}\prod_{k=1}^s |H_k|^{n_k}.\label{eqHsum}
    \end{align}
    \footnotetext{If $\kappa = s'$ this becomes $\exp\left(Q\log(n)^{1 + 3\delta/4}\right)|H_\kappa|^n$, which also satisfies the next inequality.}
    
    By the left inequality of (\ref{eqlemma7}), then $S_k \neq \emptyset$ if $\kappa \leq k \leq s$. Using the right inequality of (\ref{eqlemma7}) and that $|\hat R_k| \leq r$, then $|Q_k| \leq \frac{40r}{w_n}$ if $\kappa < k \leq s$. Adding up the bounds from the right inequality of (\ref{eqlemma7}), we have \begin{equation}\label{eqQ}
        Q = \sum_{k=\kappa}^s|Q_k| \leq \begin{cases}
        \frac{40rs}{w_n} &\text{ if } T_\kappa \neq \emptyset \\
        \frac{40rs}{w_n} + 1 & \text{ if } T_\kappa = \emptyset.
    \end{cases}
    \end{equation}
    If $r \geq w_n = \floor{\log(n)^{1+\delta}}$, plugging (\ref{eqQ}) into (\ref{eqHsum}) and using that $|H_\kappa| \leq |H_{s'}|$ yields that Lemma \ref{lemmabandcount} holds with $C = 41s$. With $r < w_n$ and $T_\kappa \neq \emptyset$, this also yields that Lemma \ref{lemmabandcount} holds with $C = 40s$.

    So, let $r < w_n$ and $T_{\kappa} = \emptyset$. Then $n_{\kappa} \leq | R_{\kappa} \cup T_{\kappa}| = |R_{\kappa}| \leq r<w_n$. By how we ordered the subgroups $H_k$ at the beginning of the proof of Lemma \ref{lemmabandcount}, if $|H_k| = |H_{\kappa}|$ then $n_{\kappa} \geq n_k$. So, there are at most $sn_\kappa \leq sw_n$ columns of depth $[G:H_\kappa]$. As there are at least $n-w_n/2 > sw_n$ columns of depth at least $[G:H_{s'}]$, then $|H_{s'}| > |H_\kappa|$. So $(|H_\kappa|/|H_{s'}|) \leq 1/p$. By (\ref{eqQ}), $Q \leq 41s$. Using that $s\log(n)^{1 + 3\delta/4} \in o(w_n)$, (\ref{eqHsum}) (with $n$ sufficiently large) becomes \[\#\Sur_{\{n_k\},r}(V,G)\leq \exp\left(41s\log(n)^{1 + 3\delta/4}\right) \cdot\left(\frac{1}{p}\right)^{w_n/2}\prod_{k=1}^s |H_k|^{n_k} \leq \prod_{k=1}^s |H_k|^{n_k},\] and the lemma follows.
\end{proof}

\subsection{Putting it all together}

When $r \in \Omega(\log(n)^{\lambda})$, Proposition \ref{mainprop}(b) follows quickly from Lemma \ref{lemmaprobsimple} and Lemma \ref{lemmabandcount}. When $r < \frac{1}{8}\log(n)^{\lambda}$, we use Lemma \ref{lemmaregions} to see that $F$ has proper robust image for all columns on which it is not a code. This means that $F$ has depth 1 for almost all columns of $M$. Using the bounds from Lemma \ref{lemmaprobsimple} and Lemma \ref{lemmanoncodes} yields Proposition~\ref{mainprop}(b).

\begin{proof}[Proof of Proposition \ref{mainprop}(b)]
    Multiplying Lemma \ref{lemmaprobsimple} and Lemma \ref{lemmabandcount} yields \begin{align*}\#\Sur_{\{n_k\},r}(V,G) \max_{F \in \Sur_{\{n_k\},r}(V,G)} \Big(\P(FM = 0)\Big) &\leq K(1-\varepsilon)^r\exp\left(\frac{Cr}{\log (n)^{\delta/4}}\right) \\
    & \leq K \exp\left(r\left(\log(1-\varepsilon) + \frac{C}{\log (n)^{\delta/4}}\right)\right).
    \end{align*}
    As $\log(1-\varepsilon) < 0$, for $n$ sufficiently large the right side of the above is bounded above by $K \exp\left(\frac{r}{2}\log(1-\varepsilon)\right)$. When $r \geq \frac{1}{8}\log(n)^{\lambda}$, this yields Proposition \ref{mainprop}(b).

    If $r < \frac{1}{8}\log(n)^{\lambda}$, then the work in the proof of Lemma~\ref{lemmaregions} gives that $T_k = \emptyset$ for $1 \leq k < s$, as ${m(T_k) \leq \frac{4r}{\log(n)^{\lambda}} < 1}$. This means that $F \in \Sur_{\{n_k\},r}(V,G)$ must have proper robust image for all columns for which it has robust image $H_k$ for $k<s$. So, $n_s \geq n - r$. Lemma \ref{lemmaprobsimple} gives that $\P(FM = 0) \leq K |G|^{r-n}$. As $n_s < n$, then there exists $j \in [n]$ and $k < s$ such that $F$ has robust image $H_k$ for $M_j$. As $[G:H_k] > 1$, Lemma~\ref{lemmanoncodes} gives  that $\#\Sur_{\{n_k\},r}(V,G) \leq K|G|^n\exp(-c\log(n)^{1+\delta})$. Using that $r \in o(\log(n)^{1+\delta})$, \[\#\Sur_{\{n_k\},r}(V,G) \max_{F \in \Sur_{\{n_k\},r}(V,G)} \Big(\P(FM = 0)\Big)  \leq K |G|^{r}\exp\left(-c \log(n)^{1+\delta}\right) \leq K \exp\left(-\frac{c}{2} \log(n)^{1+\delta}\right),\] and the proposition follows.
\end{proof}

\section*{Acknowledgements}

I am deeply grateful to Melanie Matchett Wood for suggesting this topic and advising my whole research process, in particular for her many invaluable suggestions and generous mentorship. I thank Hyungmin Jang, Nathan Kaplan, Jungin Lee, Nikita Lvov, Jiahe Shen, and Roger Van Peski for helpful comments on an earlier draft. This research was partially supported by the MIT Math Department and partially supported by the grant NSF DMS-2140043. ChatGPT noticed that an earlier version of Lemma \ref{lemma7} needed a slightly weaker conclusion if $S = \emptyset$.

\bibliography{bibliography}{}
\bibliographystyle{plain}

\end{document}